\documentclass{amsart} 

\usepackage{amsmath,amssymb,amsthm} 
\usepackage{graphicx} 
\usepackage{subfig}
\usepackage{tikz}
\usepackage[arrow, matrix, curve]{xy} 
\usepackage{color} 
\usepackage{makecell} 
\usepackage{multirow} 
\usepackage{enumerate}
\usepackage{booktabs} 
\usepackage{siunitx} 
\usepackage{enumitem} 
\usepackage{hyperref} 

\calclayout
\DeclareMathOperator{\lk}{lk}

\newcommand{\Z}{\mathbb{Z}}
\newcommand{\Q}{\mathbb{Q}}
\newcommand{\N}{\mathbb{N}}

\newcommand{\xist}{\xi_{\mathrm{st}}}
\DeclareMathOperator{\tb}{tb}
\DeclareMathOperator{\rot}{rot}
\DeclareMathOperator{\de}{d}

\makeatletter
\newcommand{\Vast}{\bBigg@{2.5}} 
\makeatother

\makeatletter
\newtheorem*{rep@theorem}{\rep@title}
\newcommand{\newreptheorem}[2]{%
\newenvironment{rep#1}[1]{%
 \def\rep@title{#2 \ref{##1}}%
 \begin{rep@theorem}}%
 {\end{rep@theorem}}}
\makeatother

\newtheoremstyle{thm}{}{}{\itshape}{}{\bfseries}{}{ }{} 
\newtheoremstyle{definition}{}{}{}{}{\bfseries}{}{ }{} 

\theoremstyle{thm}
\newtheorem{Theorem}{Theorem}[section]
\newtheorem{thm}[Theorem]{Theorem}
\newreptheorem{thm}{Theorem}  
\newtheorem{lem}[Theorem]{Lemma}
\newreptheorem{lem}{Lemma}  
\newtheorem{prop}[Theorem]{Proposition}
\newtheorem{cor}[Theorem]{Corollary}
\newtheorem*{Theorem-ohne}{Theorem}
\newtheorem{con}[Theorem]{Conjecture}

\newtheorem{ques}[Theorem]{Question}

\theoremstyle{definition}
\newtheorem{defi}[Theorem]{Definition}
\newtheorem{rem}[Theorem]{Remark}

\begingroup\expandafter\expandafter\expandafter\endgroup
\expandafter\ifx\csname pdfsuppresswarningpagegroup\endcsname\relax
\else
\fi

\definecolor{amaranth}{rgb}{0.9, 0.17, 0.31} 
\definecolor{carrotorange}{rgb}{0.93, 0.57, 0.13} 
\definecolor{citrine}{rgb}{0.89, 0.82, 0.04} 
\definecolor{dartmouthgreen}{rgb}{0.05, 0.5, 0.06} 
\definecolor{ballblue}{rgb}{0.13, 0.67, 0.8} 
\definecolor{ceruleanblue}{rgb}{0.16, 0.32, 0.75} 
\definecolor{amethyst}{rgb}{0.6, 0.4, 0.8} 
\definecolor{amber}{rgb}{1.0, 0.75, 0.0} 
\definecolor{burlywood}{rgb}{0.87, 0.72, 0.53} 

\numberwithin{equation}{section}

\begin{document}


\title{Knots that share four surgeries} 

\author{Marc Kegel}
\address{Humboldt-Universit\"at zu Berlin, Rudower Chaussee 25, 12489 Berlin, Germany.}
\email{kegelmarc87@gmail.com}

\author{Lisa Piccirillo}
\address{University of Texas at Austin, 2515 Speedway, Austin TX, USA.}
\email{lisa.piccirillo@austin.utexas.edu}


\date{\today} 

\begin{abstract}
    Distinct knots $K, K'$ can sometimes share a common $p/q$-framed Dehn surgery. A folk conjecture held that for a fixed pair of knots, this can occur for at most one value of $p/q$. We disprove this conjecture by constructing pairs of distinct knots $K,K'$ that have common Dehn surgeries for four distinct slopes. 
    We also construct non-isotopic Legendrian knots $K,K'$ that have contactomorphic contact $(+1)$- and $(-1)$-surgeries, disproving an analogous conjecture in contact geometry.
\end{abstract}

\keywords{Dehn surgery, characterizing slopes, knot traces, knot invariants, Legendrian knots, contact surgery} 

\makeatletter
\@namedef{subjclassname@2020}{%
  \textup{2020} Mathematics Subject Classification}
\makeatother

\subjclass[2020]{57R65; 57K10, 57R65, 57R58, 57K16, 57K14, 57K32} 


\maketitle


Dehn surgery for a fixed rational slope $p/q$ can be thought of as a function from the set of knots in $S^3$ to the set of closed, oriented 3-manifolds. This foundational connection between knot theory and 3-manifold topology underlies many significant results in both areas, and motivates the study of these Dehn surgery functions themselves. This paper is concerned with the injectivity of these functions. 

\vspace{5pt}
Generically, the topology of any $p/q$-Dehn surgery closely resembles the topology of the knot complement, and the more slopes one considers, the more tightly the knot complement is constrained. It is surprising for distinct knots $K,K'$ to have even a single common $p/q$-Dehn surgery. However, such $K, K'$ and $p/q$ do exist; this was originally proven in \cite{Lickorish:surgery}. 

\vspace{5pt}
This contrasts to a folk theorem\footnote{The authors are unaware whether this appears in writing, for discussion see \cite[Section 1.7]{BKM_alt}.} that for any infinite collection of rationals $R$, a pair of distinct knots $K,K'$ cannot have common $r_i$ surgeries for all $r_i\in R$. Thus arises the question: how many surgeries can a pair of distinct knots share? The (folk) conjectured answer was one.   
The main result of this work is to disprove this.

\begin{thm} \label{thm:counterex}
	For every integer $n\in\Z\setminus\{-1,0,1\}$ there exist distinct knots $K_n$ and $K'_n$ with common $r$-framed Dehn surgery for all $r\in\{n+1,n+2,n+3,2n+2\}$.
\end{thm}

For a more precise statement, we refer to Theorem~\ref{thm:main_detailed}. An example of a family of knots as in Theorem~\ref{thm:counterex} that share four surgeries is shown in Figure~\ref{fig:example}. Curiously, our $K_n$ and $K_n'$ have many other commonalities; they have the same 2-fold branched covers (Theorem~\ref{thm:covers}), the same knot Floer and Khovanov homologies, as well as many common hyperbolic invariants \cite{data}. This may be of independent interest. 

\begin{figure}[htbp] 
	\centering
	\def\svgwidth{0,85\columnwidth}
\begingroup%
  \makeatletter%
  \providecommand\color[2][]{%
    \errmessage{(Inkscape) Color is used for the text in Inkscape, but the package 'color.sty' is not loaded}%
    \renewcommand\color[2][]{}%
  }%
  \providecommand\transparent[1]{%
    \errmessage{(Inkscape) Transparency is used (non-zero) for the text in Inkscape, but the package 'transparent.sty' is not loaded}%
    \renewcommand\transparent[1]{}%
  }%
  \providecommand\rotatebox[2]{#2}%
  \newcommand*\fsize{\dimexpr\f@size pt\relax}%
  \newcommand*\lineheight[1]{\fontsize{\fsize}{#1\fsize}\selectfont}%
  \ifx\svgwidth\undefined%
    \setlength{\unitlength}{576.9648348bp}%
    \ifx\svgscale\undefined%
      \relax%
    \else%
      \setlength{\unitlength}{\unitlength * \real{\svgscale}}%
    \fi%
  \else%
    \setlength{\unitlength}{\svgwidth}%
  \fi%
  \global\let\svgwidth\undefined%
  \global\let\svgscale\undefined%
  \makeatother%
  \begin{picture}(1,0.36182006)%
    \lineheight{1}%
    \setlength\tabcolsep{0pt}%
    \put(0,0){\includegraphics[width=\unitlength,page=1]{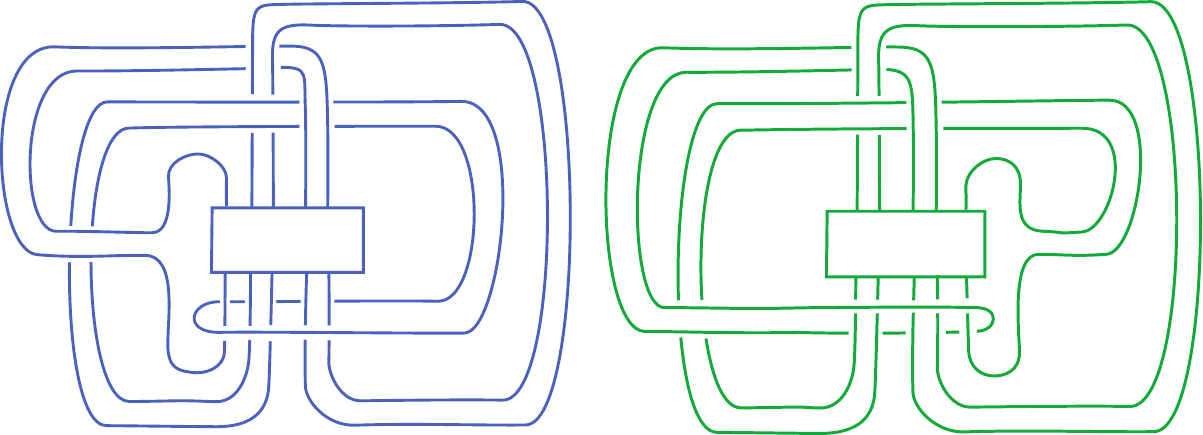}}%
    \put(0.74418369,0.15273825){\color[rgb]{0.07058824,0.66666667,0.21960784}\makebox(0,0)[lt]{\lineheight{1.25}\smash{\begin{tabular}[t]{l}$n$\end{tabular}}}}%
    \put(0.23274186,0.15485908){\color[rgb]{0.3254902,0.4,0.75686275}\makebox(0,0)[lt]{\lineheight{1.25}\smash{\begin{tabular}[t]{l}$n$\end{tabular}}}}%
  \end{picture}%
\endgroup%

	\caption{Two non-isotopic knots $K_n$ and $K'_n$ that share four integer surgeries. The box with the $n$ denotes $n$-full twists.}
	\label{fig:example}
\end{figure}
\vspace{5pt}
We don't expect that four slopes is maximal; we conjecture that there is no bound on the number of slopes a pair of distinct knots can share. Some restrictions on which pairs of knots can have many common surgeries were recently given in~\cite{BKM_alt}, which shows that for any given distinct knots $K$ and $K'$ there is an explicit constant $C(K,K')$ such that for any slope $p/q\in\Q$ with $|q|>C(K,K')$ the $p/q$-surgeries are not homeomorphic.

Injectivity questions about Dehn surgery functions also arise in the context of Legendrian knot theory. In~\cite{Casals_Etnyre_Kegel} it was conjectured that two Legedrian knots in the standard contact $3$-sphere $(S^3,\xist)$ with contactomorphic contact $(+1)$- and contact $(-1)$-surgeries are necessarily Legendrian isotopic. 
We also construct counterexamples to this conjecture.

\begin{thm} \label{thm:counterex_Legendrian}
	There exist non-isotopic Legendrian knots $K$ and $K'$ in $(S^3,\xist)$ with contactomorphic contact $r$-surgery for all $r\in\{-1,+1\}$.
\end{thm}

An example of a pair of such Legendrian knots $K$ and $K'$ is shown in Figure~\ref{fig:LegendrianHopfRBGexample}.

\subsection*{Outline}
In Section~\ref{sec:construction} we will give a method of constructing knots that share four surgeries using RBG links. The casual reader need not read past this section.
In Section~\ref{sec:distinguish}, we start the discussion of distinguishing such pairs of knots. We will discuss a long list of invariants (double branched covers, knot Floer and Khovanov homology groups, and hyperbolic invariants such as volume, cusp shape, and systole) that do not appear to distinguish our knots. We will also discuss some invariants which, via computer experimentation, apparently do (HOMFLYPT polynomials, the canonical triangulations, or the number of covers of a fixed degree of the knot exterior).
Theorem~\ref{thm:counterex_Legendrian} is proven in Section~\ref{sec:legendrian}. 
In Section~\ref{sec:questions} we give some open problems.
In Appendix~\ref{app}, we actually distinguish our pairs using HOMFLYPT polynomials. In fact, we develop a general method for computing parts of the HOMFLYPT polynomial of infinite families of knots that differ by twisting some number of strands, which may be of independent interest. 

\subsection*{Code and data}
Computations and some additional code and data can be accessed at the author's webpages~\cite{data} or can be downloaded as additional files from the arXiv version of this article. For the computations, we use code and data from \cite{Regina,SnapPy,KLO,sagemath,KnotJob,knotatlas,Searching_ribbon,Gridlink,Celoria_Grid,Dunfield_census,Szabo_calculator,Dunfield_exterior_to_link,Beker_Kegel_braid_positive,BKM_alt,BKM_QA,Kegel_Weiss,ABG+19,volt}.

\subsection*{Conventions}
We work in the smooth category. All manifolds, maps, and ancillary objects are assumed to be smooth. When we say manifolds are `the same' we mean orientation preserving diffeomorphic. Surgery coefficients of smooth links in $S^3$ are measured with respect to the Seifert longitude. Surgery coefficients of Legendrian knots are with respect to the contact longitude. 

\subsection*{Acknowledgments}

We are happy to thank Ken Baker, John Luecke, Duncan McCoy, Burak Ozbagci, and Claudius Zibrowius for useful comments and their interest in this work. Special thanks go to Jonathan Spreer for his contributions to writing code that searches for hyperbolic knots in closed manifolds. 

\subsection*{Grant support}

MK is funded by the DFG, German Research Foundation, (Project: 561898308). LP was supported in part by the Clay Foundation, the Sloan Foundation, and the Simons Collaboration grant `New Structures in Low-Dimensional Topology'. 

\section{Constructing knots that share 4-surgeries}\label{sec:construction}
In this section, we use RBG links to construct pairs of knots that share four integral surgeries. RBG links were introduced in \cite{LPshake} and are usually used to produce a pair of knots with a single common surgery.

\begin{defi}
A $3$-component link $L$ in $S^3$ with components $R$, $B$, and $G$, is called \textit{Hopf RBG link} if the following hold:
\begin{enumerate}
	\item $R\cup B$, $R\cup G$, and $B\cup G$ are Hopf links.
	\item The Seifert disks $D_R$, $D_B$, and $D_G$ of $R$, $B$, and $G$ can be isotoped such that the following three conditions hold simultaneously:
	\begin{enumerate}
		\item $D_B$ and $D_G$ intersect in a single clasp singularity,
		\item $D_B$ intersects $D_R$ in a single clasp singularity and finitely many ribbon singularities, and
		\item $D_G$ intersects $D_R$ in a single clasp singularity and finitely many ribbon singularities.
	\end{enumerate}
\end{enumerate}
\end{defi}

\begin{figure}[t] 
	\centering
	\def\svgwidth{0,7\columnwidth}
	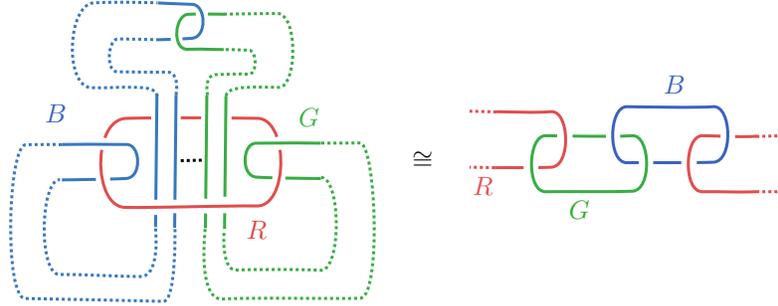
	\caption{Two schematic views on a Hopf RBG link.}
	\label{fig:RGBabstract}
\end{figure}

A schematic picture of a Hopf RBG link is shown in Figure~\ref{fig:RGBabstract}. For a concrete simple example of a (non-trivial) Hopf RBG link we refer to Figure~\ref{fig:exampleHopfRBG}. After an isotopy, we can always assume that $D_B$ intersects $D_R$ in a single clasp singularity or that $D_G$ intersects $D_R$ in a single clasp singularity. But in general, both are not possible simultaneously, see for example Figure~\ref{fig:exampleHopfRBG}.

We will show that every Hopf RBG link induces pairs of knots that share four surgeries. For that, we will need the following lemma.

\begin{lem}\label{lem:S3surgeries}
	Surgery on a Hopf link with surgery coefficients $p/q$ and $r/s$ yields $S^3$ if and only if $pr-qs=\pm1$. 
\end{lem}

\begin{proof}
	Any Dehn surgery on the Hopf link yields a lens space. Among lens spaces, we can recognize $S^3$ by the vanishing of its first homology group. Since $|pr-qs|$ yields the order of the first homology of a $(p/q,r/s)$-surgery on a $2$-component link, the claim follows.
\end{proof}

\begin{cor} \label{cor:knots} To any Hopf RBG link $L$, pair of rationals $p/q$ and $r/s$  with $pr-qs=\pm1$, and $C\in\{R,B,G\}$, we can naturally associate a knot $K_C^{p/q,r/s}$ by considering the image of $C$  in the $S^3$ obtained by performing $p/q$ and  $r/s$ surgery on $L\setminus C$. \qed
\end{cor}

In this paper, we will mainly be concerned with the knots $K_G^{0,n}$ and $K_B^{0,n}$.

\noindent By convention, we always choose an orientation of the Hopf RBG link such that
\begin{align*}
	\lk(R,B)=1=\lk(R,G).
\end{align*}
Thus the linking number $\lk(B,G)$ is either $+1$ or $-1$. Then the precise version of our main theorem reads as follows. 

\begin{thm} \label{thm:main_detailed}
	For any Hopf RBG link L and integer $n\in\Z$, the knots $K_G^{0,n}$ and $K_B^{0,n}$ have orientation-preserving diffeomorphic surgeries for the following set of surgery slopes.
	\begin{itemize}
	\item $(n+1,n+2,n+3,2n+2)$, if $\lk(B,G)=-1$, and
	\item $(n-3,n-2,n-1,2n-2)$, if $\lk(B,G)=1$.
	\end{itemize}
\end{thm}

To prove Theorem~\ref{thm:main_detailed}, we first observe that there are distinct fillings of a Hopf RBG link which produce the same manifold. We use the notation $L(r,b,g)$ to denote filling $R,B,G$ with slopes $r,b,g$, respectively. We use $\ast$ to denote no filling.  

\begin{lem}\label{lem:sym_hopf}
	Let $L$ be a Hopf RBG link. Then for any $n,m\in\Z$ and $s\in\{-1,0,1\}$ we have that $L(n,m,s)$ is orientation-preserving diffeomorphic to $L(n,s,m)$. 
\end{lem}

\begin{proof}
	For each $s$, these diffeomorphisms are induced by the sequences of handle slides and slam dunks shown in Figures~\ref{fig:lem1} and~\ref{fig:lem2}.
\end{proof}

\begin{figure}[t] 
	\centering
	\def\svgwidth{0,8\columnwidth}
	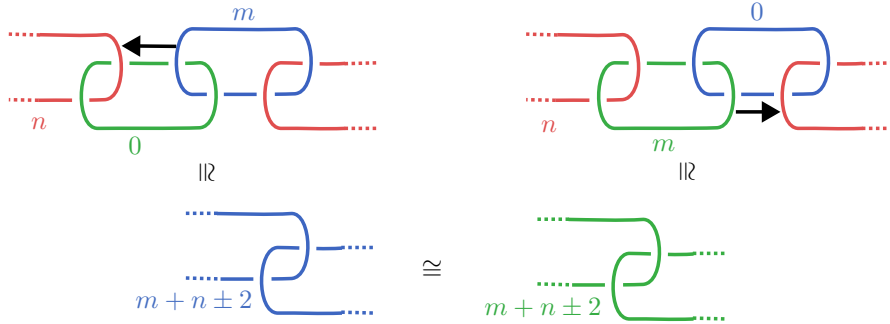
	\caption{$L(n,m,0)$ is diffeomorphic to $L(n,0,m)$. Each vertical diffeomorphism is given by a handle slide (indicated with the arrow) followed by a slam dunk.}
	\label{fig:lem1}
\end{figure}

\begin{figure}[htbp] 
	\centering
	\def\svgwidth{0,75\columnwidth}
	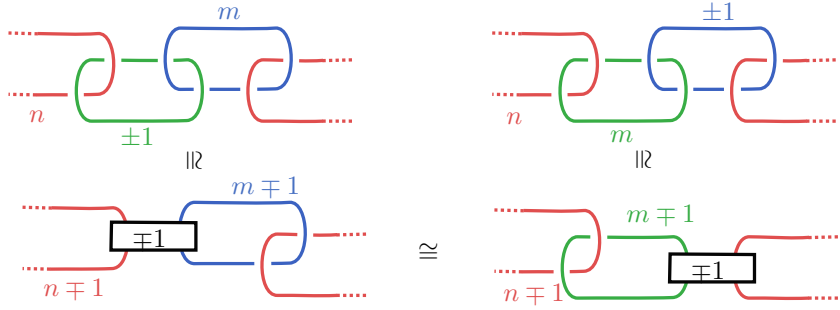
	\caption{$L(n,m,\pm1)$ is diffeomorphic to $L(n,\pm1,m)$. Each vertical diffeomorphism is given by a blow down. While the horizontal diffeomorphism is induced by an isotopy.}
	\label{fig:lem2}
\end{figure}
\begin{figure}[htbp] 
	\centering
	\def\svgwidth{0,75\columnwidth}
	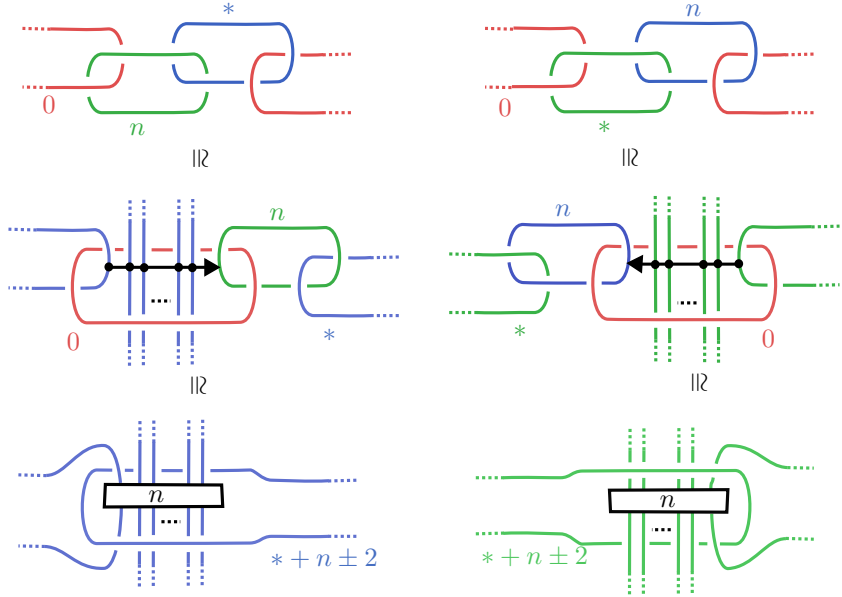
	\caption{Kirby moves from a Hopf RBG link to the knots $K_B^{0,n}$ and $K_G^{0,n}$. The sign in the framing $*+n\pm2$ depends on the exact linking numbers of the Hopf RBG link.}
	\label{fig:RGBtangles}
\end{figure}

With this lemma, we are ready to give the proof of Theorem~\ref{thm:main_detailed}.

\begin{proof} [Proof of Theorem~\ref{thm:main_detailed}]
For knots $K_G^{0,n}$ and $K_B^{0,n}$ constructed from a Hopf RBG link $L$ as in Corollary \ref{cor:knots}, the diffeomorphisms from	Lemma~\ref{lem:sym_hopf} readily show that  $K_G^{0,n}$ and $K_B^{0,n}$ share four integer surgeries as follows.
	\begin{align*}
		K_G^{0,n}(s_n)&=L(0,n,n)=K_B^{0,n}(s_n)\\
		K_G^{0,n}(s_0)&=L(0,n,0)=L(0,0,n)=K_B^{0,n}(s_0)\\
		K_G^{0,n}(s_1)&=L(0,n,+1)=L(0,+1,n)=K_B^{0,n}(s_1)\\
		K_G^{0,n}(s_{-1})&=L(0,n,-1)=L(0,-1,n)=K_B^{0,n}(s_{-1})
	\end{align*}
Here we are using $s_i$ to stand in for a not-yet-computed surgery coefficient. To compute the coefficients $(s_n,s_0,s_1,s_{-1})$ we perform the explicit Kirby moves as shown in Figure~\ref{fig:RGBtangles}. A concrete Hopf RBG link is shown in Figure~\ref{fig:exampleHopfRBG}, where $(s_n,s_0,s_1,s_{-1})=(n+1,n+2,n+3,2n+2)$. To get the other set of slopes in the statement of Theorem \ref{thm:main_detailed}, one can switch the sign on the clasp between $B$ and $G$ in the RBG link in Figure~\ref{fig:exampleHopfRBG}.
\end{proof}

\begin{rem}\label{rem:tangles} A few remarks are in order.
	\begin{enumerate}
		\item If the Hopf RBG links are overly simple or symmetric then it might happen that the knots $K_G^{0,n}$ and $K_G^{0,n}$ are isotopic. But in general, there is no reason why these knots should be isotopic, and we will prove in Section~\ref{sec:distinguish} that often they are not. 
		\item On the other hand for $n=-1,0,1$, the knots $K_G^{0,n}$ and $K_B^{0,n}$ are always isotopic. Indeed, Lemma~\ref{lem:sym_hopf} implies that
		\begin{align*}
			K_G^{0,0}&=L(0,0,*)=L(0,*,0)=K_B^{0,0},\\
			K_G^{0,+1}&=L(0,+1,*)=L(0,*,+1)=K_B^{0,+1},\\
			K_G^{0,-1}&=L(0,-1,*)=L(0,*,-1)=K_B^{0,-1}.
		\end{align*}
		\item Similarly for all integers $n\in\Z$, the knots $K_G^{n,0}$ and $K_B^{n,0}$ are isotopic.
  \item It follows from Figure~\ref{fig:RGBtangles} that the knots $K_B^{0,n}$ and $K_G^{0,n}$ are related by taking the tangle shown in the lower row of Figure~\ref{fig:RGBtangles} and rotating it by $\pi$. This can be thought of as a generalized mutation. 
	\end{enumerate}
\end{rem}

In Theorem~\ref{thm:main_detailed} we established that the knots $K_G^{0,n}$ and $K_B^{0,n}$ have four common surgeries. Recall that for any $s\in\Z$ we can associate a 4-manifold to a knot $K$ by attaching a $2$-handle with framing $s$ along a knot $K$ in $\partial D^4$. This $4$-manifold is called the $s$-trace, $X_s(K)$. The boundary of $X_s(K)$ is diffeomorphic to the $s$-surgery along $K$. Since the RBG link which induced $K_G^{0,n}$ and $K_B^{0,n}$ had $R=U$ and $r=0$, our Theorem \ref{thm:main_detailed} immediately generalizes to traces by taking the RBG link $L$ to be a handle diagram of a 4-manifold where $R$ is a $1$-handle in dotted circle notation. 

\begin{cor}\label{thm:traces}
	Let $L$ be a Hopf RBG link. For every integer $n\in\Z$ the knots $K_G^{0,n}$ and $K_B^{0,n}$ have orientation-preserving diffeomorphic $s_i$-traces for the following set of surgery slopes.
	\begin{itemize}
	\item $(n+1,n+2,n+3,2n+2)$, if $\lk(B,G)=-1$, and
	\item $(n-3,n-2,n-1,2n-2)$, if $\lk(B,G)=1$.\qed
	\end{itemize}
\end{cor}

\section{Distinguishing knots that share four surgeries}\label{sec:distinguish}
As a concrete example, we consider the Hopf RBG link $L$ shown in Figure~\ref{fig:exampleHopfRBG}. Note that the intersection of the Seifert disk $D_R$ of $R$ intersects the blue and red Seifert disks in two ribbon singularities and two clasp singularities. On the left side, both ribbon intersections are with $D_B$, and on the right, both are with $D_G$. By following the Kirby moves from Figure~\ref{fig:RGBtangles} we obtain the knots $K_B^{0,n}=K_n$ and $K_G^{0,n}=K'_n$ shown in Figure~\ref{fig:example}.

\begin{figure}[t] 
	\centering
	\def\svgwidth{0,75\columnwidth}
	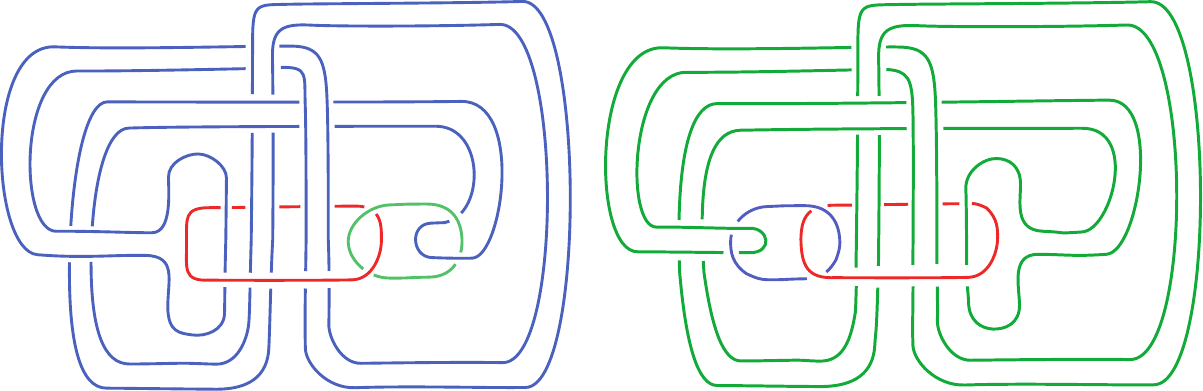
	\caption{Two views on the same Hopf RBG link.}
	\label{fig:exampleHopfRBG}
\end{figure}

In Appendix~\ref{app} we will prove a technical theorem which implies the following.

\begin{thm}\label{thm:distinguish}
 The Hopf RGB link $L$ in Figure~\ref{fig:exampleHopfRBG} induces the knots $K_B^{0,n}$ and $K_G^{0,n}$ in Figure~\ref{fig:example} which, for $n\neq-1,0,+1$, have distinct HOMFLYPT polynomials.
\end{thm}

Together with Theorem \ref{thm:main_detailed}, this finishes the proof of Theorem~\ref{thm:counterex}.

\begin{rem} \label{rem:invariants} We make a few remarks about other ways to distinguish or not distinguish knots that share four surgeries.
	\begin{enumerate}
		\item In all examples of Hopf RBG links that we have inspected, the knots $K_B^{0,n}$ and $K_G^{0,n}$ share many knot invariants. For example, they have isomorphic knot Floer and Khovanov homologies and thus also the same Alexander and Jones polynomial and the same $3$-genus. Moreover, verified computations with SnapPy suggest that such knots always have the same hyperbolic invariants (volume, cusp shape, symmetry group, systole, and short slopes). This is probably explained by the fact that the knots coming from our construction are always related by a generalized mutation, as explained in Remark~\ref{rem:tangles}. 
        \item Knots coming from our construction have the same double branched cover; see Theorem~\ref{thm:covers} below.
		\item We can also distinguish $K_B^{0,n}$ and $K_G^{0,n}$ using the fundamental groups of their complements by looking at the number of conjugacy classes of index $k$ subgroups. Indeed, for the Hopf RBG link $L$ from Figure~\ref{fig:exampleHopfRBG} we have used SnapPy to compute (for small values of $n$) the number of covers of $K_B^{0,n}$ and $K_G^{0,n}$ to be different. For example it turns out that $K_B^{0,2}$ has $32$ covers of degree $8$ while $K_G^{0,2}$ has $33$ covers of degree $8$. 
		\item Another successful and computationally fast way to distinguish such knots is by using canonical triangulations. Here we have used the verified methods in SnapPy~\cite{SnapPy} to construct the canonical triangulations of $K_B^{0,n}$ and $K_G^{0,n}$ for small values of $n$ and show that they are not combinatorial equivalent, cf.~\cite{Weeks_canonical}. With the results and methods from~\cite{canonical_surgery} it should also be possible to distinguish infinitely many $K_B^{0,n}$ and $K_G^{0,n}$ using their canonical triangulations.
		\item There exist also examples of Hopf RBG links where $D_R\cap (D_B\cup D_G)$ consists of a single ribbon and two clasp singularities such that $K_B^{0,n}$ and $K_G^{0,n}$ are non-isotopic, see~\cite{data} for a concrete example. However, then the inductive formula for the HOMFLYPT polynomial yields that for all $n$, the knots $K_B^{0,n}$ and $K_G^{0,n}$ have the same HOMFLYPT polynomial. On the other hand, the canonical triangulations can still distinguish $K_B^{0,n}$ and $K_G^{0,n}$ (for $n\neq-1,0,1$). 
	\end{enumerate}
\end{rem}

We end this section by showing that all knots coming from our construction share the same double branched cover and that the cyclic $k$-fold branched covers have the same volume. We denote by $\Sigma_k(K)$ the cyclic $k$-fold branched cover of $K$. 

\begin{thm}\label{thm:covers} 
    Let $L$ be a Hopf RBG link with associated knots $K_B^{0,n}$ and $K_G^{0,n}$. Then for every $k\geq2$ there exists a $2$-component link $J_k$ consisting of two unknots such that for all $n\in\Z$ 
    \begin{itemize}
        \item $\Sigma_k(K_B^{0,n})$ is orientation-preserving diffeomorphic to $J_k(0,kn)$, and
        \item $\Sigma_k(K_G^{0,n})$ is orientation-preserving diffeomorphic to $J_k(kn,0)$.
    \end{itemize}
    Moreover, there exists a sequence of mutations on $J_k$ interchanging the components of $J_k$, and thus the volumes of $\Sigma_k(K_B^{0,n})$ and $\Sigma_k(K_G^{0,n})$ agree for all $k$ and $n$.
    For $k=2$ there exists an isotopy of $J_2$ interchanging the components and thus for all $n\in\Z$ the double branched covers of $K_B^{0,n}$ and $K_G^{0,n}$ are orientation-preserving diffeomorphic.
    \end{thm}

\begin{figure}[t] 
	\centering
	\def\svgwidth{0,85\columnwidth}
	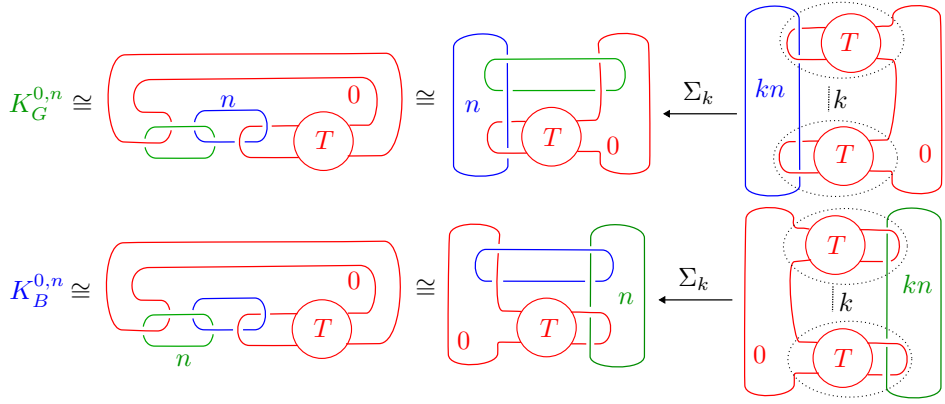
	\caption{Creating surgery diagrams of the cyclic $k$-fold covers of $K_B^{0,n}$ (bottom row) and of $K_B^{0,n}$ (top row). Here the disk labeled $T$ denotes a tangle such that the red knot $R$ is an unknot.}
	\label{fig:dbc}
\end{figure}

\begin{proof}
    The Hopf RBG link $L(0,*,n)$ is a surgery presentation of the complement of $K_B^{0,n}$ in which $K_B^{0,n}$ is represented by the blue unknot $B$. Thus we can create a surgery diagram of the cyclic $k$-fold branched cover $\Sigma_k(K_B^{0,n})$ of $K_B^{0,n}$ from that surgery presentation by taking the $k$-fold cyclic cover of $L(0,*,n)$ branched along $B$ as shown in the bottom frames of Figure~\ref{fig:dbc}. We do the same for $\Sigma_k(K_G^{0,n})$ in the top frames of Figure \ref{fig:dbc}. 

    Define $J_k$ to be the $2$-component link in the top right frame of Figure~\ref{fig:dbc}. Note that downstairs $R$ is an unknot bounding a disk $D_R$ in the diagram with only ribbon singularities, and one can check the lift of $R$ to $\Sigma_k(K_G^{0,n})$ has the same property. As such we can isotope $J_k$ to the link in the bottom right frame of Figure~\ref{fig:dbc}. The itemized claims now follow. 

    The sequence of mutations which exchange the components of $J_k$ are marked in black dashed circles in Figure~\ref{fig:dbc}. Finally, we observe that there is an isotopy of $J_2$ that interchanges the components given by moving one clasp of $J_2$ through both tangles $T$. This shows that the double branched covers are orientation-preserving diffeomorphic. 
\end{proof}

\begin{rem}
    On the other hand, it follows for example from~\cite{covers1,covers2,covers3} that for any two non-isotopic knots there exists a $k$ such that their $k$-fold cyclic covers are not orientation-preserving diffeomorphic. In particular, this holds for the knots from Figure~\ref{fig:example} which we will distinguish in Appendix~\ref{app}. On the other hand, Theorem~\ref{thm:covers} says that it might be difficult to use the cyclic covers as invariants to distinguish knots in our setting.
\end{rem}

\section{Legendrian knots that share contact \texorpdfstring{$(\pm1)$}{(+-1)}-surgeries}\label{sec:legendrian}
Here we will prove Theorem~\ref{thm:counterex_Legendrian} by a variation of the proof of Theorem~\ref{thm:counterex} in the context of contact manifolds. In this section, we will measure surgery coefficients with respect to the contact longitude. For background on contact geometry and contact surgery, we refer for example  to~\cite{Gompf_Stein,Ding_Geiges_Surgery,Ding_Geiges_Stipsicz,Ozbagci_Stipsicz_book,Geiges_book,Durst_Kegel_rot_surgery,phdthesis,Legendrian_knot_complement,Casals_Etnyre_Kegel,EKS_contact_surgery_numbers}.

\begin{defi}
	A $3$-component Legendrian link $L$ in $(S^3,\xist)$ with components $R$, $B$, and $G$, is called \textit{Legendrian Hopf RBG link} if $L$ is a smooth Hopf RBG link such that $R$, $B$, and $G$ are Legendrian unknots with $\tb=-1$.
\end{defi}

A concrete example is shown in Figure~\ref{fig:LegendrianHopfRBGexample}. Similar as in Section~\ref{sec:construction} any Legendrian Hopf RBG link yields Legendrian knots $K_B$ and $K_G$ in $(S^3,\xist)$ that have contactomorphic contact $(\pm1)$-surgeries.

\begin{lem}\label{lem:pairs_Legendrian}
	If $L$ is a Legendrian Hopf RBG link. Then 
	\begin{align*}
		K_G=&\,L(+1,-1,*) \textrm{ and }
		K_B=\,L(+1,*,-1)
	\end{align*}
	are Legendrian knots in $(S^3,\xist)$.
\end{lem}

\begin{proof}
	The cancellation lemma implies that contact $(-1)$-surgery on a Legendrian knot $K$ followed by contact $(+1)$-surgery on a Legendrian meridian of $K$ with $\tb=-1$ is contactomorphic to $(S^3,\xist)$. In particular, it follows that contact $(+1)$-surgery on $R$ and contact $(-1)$-surgery on $G$ is contactomorphic to $(S^3,\xist)$. Thus the image of $B$ under this contactomorphism is a Legendrian knot $K_B$ in $(S^3,\xist)$. The other knot arises by the same construction with the roles of $B$ and $G$ reversed.
\end{proof}

\begin{lem}\label{lem:sym_hopf_Legendrian}
	Let $L$ be a Legendrian Hopf RBG link. Then $L(+1,+1,-1)$ is contactomorphic to $L(+1,-1,+1)$. 
\end{lem}

\begin{proof}
	We consider the contactomorphism shown in Figure~\ref{fig:LegendrianLemma}. The right and left contactomorphism are given by the indicated contact handle slides and Legendrian Reidemeister moves. That the two surgery diagrams in the middle are contactomorphic can be seen as follows. First, we compute that contact $(+1)$-surgery on a Legendrian knot $K$ followed by contact $(+1)$-surgery on a Legendrian meridian of $K$ with $\tb=-1$ yields the overtwisted contact structure $\xi_1$ on $S^3$ with $\de_3=1$~\cite{Durst_Kegel_rot_surgery}. Then we use smooth Kirby calculus and the formulas for computing $\tb$ and $\rot$ from surgery diagrams~\cite{phdthesis} to show that the green knot in the second surgery diagram and the blue knot in the third surgery diagram are loose Legendrian knots in $(S^3,\xi_1)$ with the same classical invariants. Thus these two Legendrian knots are isotopic~\cite{Etnyre_OT} and then also the surgered manifolds are contactomorphic.
\end{proof}

\begin{figure}[t] 
	\centering
	\def\svgwidth{0,75\columnwidth}
	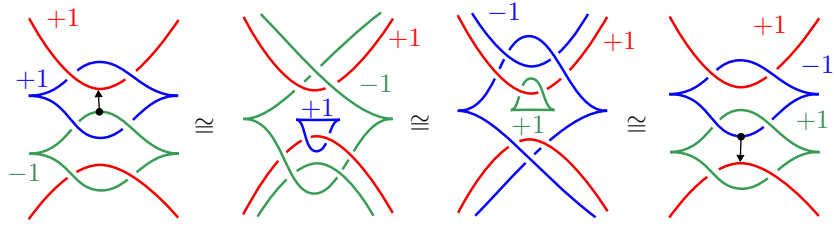
	\caption{$L(+1,+1,-1)$ is contactomorphic to $L(+1,-1,+1)$.}
	\label{fig:LegendrianLemma}
\end{figure}

\begin{proof} [Proof of Theorem~\ref{thm:counterex_Legendrian}]
	Let $L$ be the Legendrian Hopf RBG link shown in Figure~\ref{fig:LegendrianHopfRBGexample}. Then Lemma~\ref{lem:pairs_Legendrian} and~\ref{lem:sym_hopf_Legendrian} imply that
	\begin{align*}
		K_G(-1)&=L(+1,-1,-1)=K_B(-1),\,\text{ and }\\
		K_G(+1)&=L(+1,+1,-1)=L(+1,-1,+1)=K_B(+1).
	\end{align*}
Explicit contact handle slides yield front projections of the knots $K_B$ and $K_G$ shown in Figure~\ref{fig:LegendrianHopfRBGexample}. And by computing the HOMFLYPT polynomials~\cite{data} we see that these two Legendrian knots are not even smoothly isotopic.
\end{proof}

\begin{figure}[t] 
	\centering
	\def\svgwidth{0,75\columnwidth}
	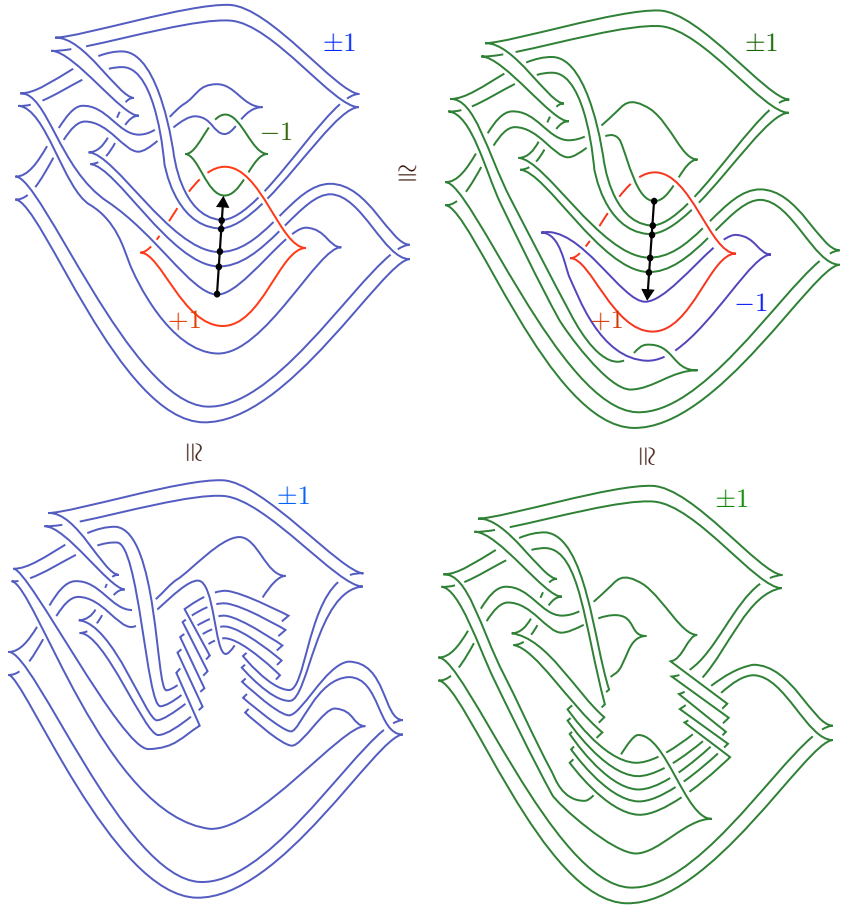
	\caption{The top row shows two views on the same Legendrian Hopf RBG link. The bottom row shows front projections of the Legendrian knots $K_B$ and $K_G$ that have contactomorphic contact $(\pm1)$-surgeries.}
	\label{fig:LegendrianHopfRBGexample}
\end{figure}

\section{Some open questions}\label{sec:questions}

We end this article with some open questions. In Remark~\ref{rem:invariants} we observed that many invariants agree on our knots that share four surgeries. We don't know how general this is. 

\begin{ques}
	Let $K$ and $K'$ be two knots that share more than one surgery. Do $K$ and $K'$ have the same knot Floer and Khovanov homology? Do they have the same hyperbolic invariants, such as volume or cusp shape?
\end{ques}

Recall that in Remark~\ref{rem:tangles}(5) we showed that our pairs of knots with four common surgeries are related by a generalized mutation operation. This might explain why our knots share many knot invariants; it might not be inherent to having many common surgeries. The effect of a different generalization of mutation on various invariants was studied in~\cite{Mutations1,Mutations2}, but our examples are not covered by these results.

We also conjecture that in fact there is no bound on the number of slopes for which some pair of knots can have common surgeries.  

\begin{con}
	For any natural number $N$ there exist pairwise-different integers $n_1,\ldots n_N$ and non-isotopic knots $K$ and $K'$ such that for $i=1,\ldots N$, $K(n_i)$ is orientation-preserving diffeomorphic to $K'(n_i)$.
\end{con}

One can also ask more specific versions of these questions, such as whether, for any finite set of rationals, one can find pairs of distinct knots which share those surgeries, or whether one can find infinite families of distinct knots with several common surgeries.  

\appendix

\section{HOMFLYPT polynomials of twist families}\label{app}

In this appendix, we will prove Theorem~\ref{thm:distinguish} by showing that for all integers $n\in \Z\setminus\{-1,0,1\}$, the HOMFLYPT polynomials of the knots $K_B^{0,n}$ and $K_G^{0,n}$ from Figure~\ref{fig:example} are different.

\subsection{The degree 0 part of the HOMFLYPT polynomial}

The HOMFLYPT polynomial $H_L(v,z)$ of an oriented link $L$ is defined via $H_{\textrm{unknot}}(v,z) = 1$ and the following skein relation
\begin{align*}
	v^{-1}H_{L+} (v,z) - vH_{L_-} (v,z) = zH_{L_0}(v,z).
\end{align*}
Here $L_+$, $L_-$, and $L_0$ are links that agree outside a $3$-ball $B$ such that $L_\pm$ admits a single positive/negative crossing inside $B$ and this crossing is resolved (following the orientation) in $L_0$.
It is known that the HOMFLYPT polynomial can be written as 
	\begin{equation*}
		H_L(v,z)=(v^{-1}z)^{1-\#L} \sum_{i\in\N_0} H^i_L(v)z^{2i},
	\end{equation*}
	where $\#L$ denotes the number of components of the link $L$ and $H^i_L(v)$ is a Laurent polynomial in $v^2,v^{-2}$. We write $F_L(v)$ for $H^0_L(v)$; this $F_L$ is characterized by $F_{\text{unknot}}(v)=1$ and the following simple skein relation:
	\begin{equation*}
		v^{-2}F_{L_+}(v)-F_{L_-}(v)=\begin{cases}
			F_{L_0}(v)\,\,&\text{ if } \delta=0,\\
			0\,\,&\text{ if } \delta = 1,
		\end{cases}
	\end{equation*}
	where $\delta=0$ if the two strands at the resolved crossing belong to the same component and $\delta=1$ otherwise. In particular, this implies that $F_L(v)$ of a link $L$ with components $K_1\cup\cdots\cup K_n$ only depends on the linking numbers $\lk(K_i,K_j)$ and the polynomials $F_{K_i}(v)$ of the components. More precisely, we have for an oriented link $L=K_1\cup\cdots\cup K_n$ with components $K_i$ the following linking number formula
	\begin{align*}
    F_L(v)=(v^{-2}-1)^{n-1} v^{2\sum_{i<j} \lk(K_i,K_j)} F_{K_1}(v)\cdots F_{K_n}(v).
	\end{align*}
	We refer to~\cite{Ito_HOMFLY0} for more details on $F_L$. 
	
	\subsection{The knots \texorpdfstring{$K_n$}{Kn} and \texorpdfstring{$K'_n$}{K'n}}
	
	In the rest of this appendix, we denote by $K_n=K_B^{0,n}$ and $K'_n=K_G^{0,n}$ the knots shown in Figure~\ref{fig:example}. The goal of this section is to prove the following theorem, which implies Theorem~\ref{thm:distinguish}.
	
	\begin{thm}\label{thm:main_homfly}
		The degree $0$ parts of the HOMFLYPT polynomials fulfill
	\begin{align*}
		F_{K_n}(v)&=\begin{cases}
			1+v^{-2n-4}+\text{ higher order terms}\,\,& \text{ if } n\leq-3\\
			1+(2n-2)v^{6-2n}+\text{ lower order terms}\,\,& \text{ if } n\geq 4
		\end{cases}\\
		F_{K'_n}(v)&=\begin{cases}
			1+(n+2)v^{-2n-4}+\text{ higher order terms}\,\,& \text{ if } n\leq-3\\
			1+(n-2)v^{6-2n}+\text{ lower order terms}\,\,& \text{ if } n\geq 4.
		\end{cases}
	\end{align*} 
For the other values of $n$, we have the results for $F$ as shown in Table~\ref{tab:tab_homfly_values}.
In particular, the HOMFLYPT polynomials of $K_n$ and $K'_n$ differ for all $n\neq0,\pm1$.
	\end{thm}

For the proof of Theorem~\ref{thm:main_homfly} we use the following strategy. $K_n$ and $K_{n-1}$ only differ by a single full twist of $5$ parallel strands. And the same holds true for $K'_n$ and $K'_{n-1}$. Using the skein relation of $F_L$ we can change crossings in the diagram of $K_{n}$ to inductively reduce the computation of $F_{K_n}$ to the computation of $F_{K_{n-1}}$ and the polynomials of knots that differ only by twists of $4$, $3$ or $2$ strands which can again be computed inductively. In the following, we make this more precise.
	
\subsection{A 5-twist inductive formula}

\begin{figure}[t] 
	\centering
	\def\svgwidth{0,5\columnwidth}
	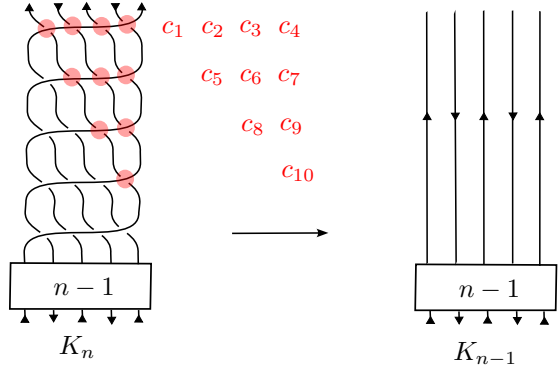
	\caption{Left: The knot $K_n$ with ten crossings $c_i$ marked. Right: Changing the ten marked crossings yields the knot $K_{n-1}$.}
	\label{fig:5strands}
\end{figure}

We start with a general formula that relates $F$ of knots that differ by twisting $5$ strands, where $3$ of these strands are oriented in one direction and the other $2$ strands are oriented in the other direction.

For all $n\in\N_0$, note that the $K_n$ are outside of a $3$-ball $B$ identical and inside $B$ have $5$ strands performing $n$-full twists, oriented as shown in Figure~\ref{fig:5strands}. For $i=1,\ldots, 10$, let $J_i$ be the links that are outside of $B$ the same as $K_n$ but are inside $B$ as shown in Figure~\ref{fig:Ji}.
\begin{lem}\label{lem:5twist}
For $n\geq1$ and $i\in\{1,\ldots, 10\}$ the $F$ polynomials of $K_n, J_i$, and $K_{n-1}$ are related by
\begin{align*}
	F_{K_n}=\,&v^2\big(F_{J_1} - F_{J_2} +F_{J_3} -F_{J_4} \big) + \big(-F_{J_5} +F_{J_6} -F_{J_7} +F_{J_8} -F_{J_9} \big) - v^{-2} F_{J_{10}} \\&+ v^{-4} F_{K_{n-1}}.
\end{align*}	
\end{lem}

\begin{proof}
		For the crossings $c_1,\ldots, c_{10}$ in $K_n$ as marked in Figure~\ref{fig:5strands} we define for $i=1,\ldots, 10$ the knot $K_n^i$ to be the knot obtained from $K_n$ by changing all crossings $c_j$ for $j\leq i$. We observe that $K_n^{10}=K_{n-1}$ and we set $K_n^0=K_n$.
		Moreover, the links $J_i$ are $2$-component links obtained from $K_n$ by changing all crossings $c_j$ for $j>i$ and resolving the crossing $c_i$.
		By appealing to the Skein tree in Figure \ref{fig:skeintree} 
        we get the claimed formula.
\end{proof}

\begin{rem}\hfill
\begin{enumerate}
    \item The lemma also holds for $K_n', J_i'$, defined analogously. In fact, it holds for any sequence of knots $K_n$ that differ by twisting $5$-strands as shown in Figure~\ref{fig:5strands}.
    \item 	We get a similar formula for negative twists by using the mirror formula for $F_L$, i.e.\ $F_{-L}(v)=F_L(v^{-1})$, where $-L$ denotes the mirror of $L$.
\end{enumerate}
\end{rem}

        \begin{figure}[t]
            \centering 
            \resizebox{\textwidth}{!}{\begin{tikzpicture}
    \begin{scope}
        \node (Kn) at (0,0) {$K_n^0$};
            \node (K1) at (2,0) {$K_n^1$};
            \node (K2) at (4,0) {$K_n^2$};
            \node (K3) at (6,0) {$K_n^3$};
            \node (K4) at (8,0) {$K_n^4$};
            \node (K5) at (10,0) {$K_n^5$};
            \node (K6) at (12,0) {$K_n^6$};
            \node (K7) at (14,0) {$K_n^7$};
            \node (K8) at (16,0) {$K_n^8$};
            \node (K9) at (18,0) {$K_n^9$};
            \node (K10) at (20,0) {$K_n^{10}$};
        \node (J1) at (0,-2) {$J_1$};
        \node (J2) at (2,-2) {$J_2$};
        \node (J3) at (4,-2) {$J_3$};
        \node (J4) at (6,-2) {$J_4$};
        \node (J5) at (8,-2) {$J_5$};
        \node (J6) at (10,-2) {$J_6$};
        \node (J7) at (12,-2) {$J_7$};
        \node (J8) at (14,-2) {$J_8$};
        \node (J9) at (16,-2) {$J_9$};
        \node (J10) at (18,-2) {$J_{10}$};

    \end{scope}
    
    \begin{scope}
        \path [->] (Kn) edge node [anchor=south] {$v^2$} (K1);
        \path [->] (K1) edge node [anchor=south] {$v^{-2}$} (K2);
        \path [->] (K2) edge node [anchor=south] {$v^{2}$} (K3);
        \path [->] (K3) edge node [anchor=south] {$v^{-2}$} (K4);
        \path [->] (K4) edge node [anchor=south] {$v^{-2}$} (K5);
        \path [->] (K5) edge node [anchor=south] {$v^{2}$} (K6);
        \path [->] (K6) edge node [anchor=south] {$v^{-2}$} (K7);
        \path [->] (K7) edge node [anchor=south] {$v^{2}$} (K8);
        \path [->] (K8) edge node [anchor=south] {$v^{-2}$} (K9);
        \path [->] (K9) edge node [anchor=south] {$v^{-2}$} (K10);

        \path [->] (Kn) edge node [anchor=west] {$v^{2}$} (J1); 
        \path [->] (K1) edge node [anchor=west] {$-1$} (J2); 
        \path [->] (K2) edge node [anchor=west] {$v^{2}$} (J3);
        \path [->] (K3) edge node [anchor=west] {$-1$} (J4);
        \path [->] (K4) edge node [anchor=west] {$-1$} (J5);
        \path [->] (K5) edge node [anchor=west] {$v^{2}$} (J6);
        \path [->] (K6) edge node [anchor=west] {$-1$} (J7);
        \path [->] (K7) edge node [anchor=west] {$v^{2}$} (J8);
        \path [->] (K8) edge node [anchor=west] {$-1$} (J9);
        \path [->] (K9) edge node [anchor=west] {$-1$} (J10);

        

    \end{scope}

\end{tikzpicture}} 
            \caption{A Skein tree for relating the F polynomials of $K_n$ and $K_{n-1}$.} 
            \label{fig:skeintree} 
        \end{figure}
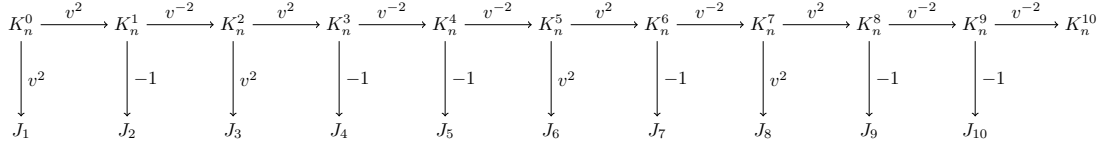 

\subsection{Proof of Theorem~\ref{thm:main_homfly}}

Using Lemma~\ref{lem:5twist} we can prove Theorem~\ref{thm:main_homfly}. For that, we need a computation of the $F$ polynomials for the $J_i$.

\begin{prop}\label{prop:Jipolys}
    The $F$ polynomials of the links $J_i$ and $J'_i$ are as given in Table \ref{tab:tab_homfly_values}.
\end{prop}

We will prove Proposition \ref{prop:Jipolys} for $i=1$ in Section \ref{sec:J1}. We suppress the proof for $i>1$; for all $i>1$ the proof is similar to and simpler than the proof for $i=1$. 

We will also need the following lemma, which can be proven via a standard induction argument.

\begin{lem}\label{lem:ana1}
	Let $X$, $Y$, and $Z$ be polynomials in $v$. We define a sequence of polynomials $F_n$ via the recursive formula
	\begin{align*}
		F_n=X+Yv^{-2n}+Zv^{-4n}+v^{-4}F_{n-1}.
	\end{align*}
Then for all $n>n_0\geq0$ we have
	\begin{align*}
	F_n=X\frac{1-v^{-4(n-n_0)}}{1-v^{-4}}+Yv^{-2n}\frac{v^{-2(n-n_0)}-1}{v^{-2}-1}+(n-n_0)Zv^{-4n}+v^{-4(n-n_0-1)}F_{n_0}\qed
	\end{align*}
\end{lem}

\begin{proof}[Proof of Theorem~\ref{thm:main_homfly}]
	We will prove the theorem for positive $n$. For negative $n$, we use the mirror formula for $F_L$ and proceed analogously. 
	First, we notice that we can orient the knots $K_n$ from Figure~\ref{fig:example} such that we can apply Lemma~\ref{lem:5twist}. Below we explain how to compute the $F$ polynomials of the $2$-component links $J_i$. In particular, we get for all $n\geq6$ the values shown in Table~\ref{tab:tab_homfly_values}.

Plugging in the values for $F_{J_i}$ into Lemma~\ref{lem:5twist} we get for all $n\geq6$
\begin{align*}
	 F_{K_n}=\,&X+Yv^{-2n}+Zv^{-4n}+ v^{-4} F_{K_{n-1}} 
\end{align*}
for polynomials $X$, $Y$, and $Z$ given by
{\scriptsize
\begin{align*}
	X=&\,1-v^{-4}\\
	Y=&\,-v^{-6}+(4-2n)v^{-4}+(10n-8)v^{-2}+(12-20n)+(20n-13)v^{2}+(8-10n)v^{4}+(2n-2)v^{6}\\
	Z=&\,-(2+8n)v^{-4}+38nv^{-2}+(17-70n)+(60n-28)v^{2}+(12-20n)v^{4}+(4-2n)v^{6}+(2n-3)v^{8}.
\end{align*}}From Lemma~\ref{lem:ana1} it follows that for all $n\geq6$, we have
\begin{align*}
	F_{K_n}=&\,1-v^{-4(n-5)}+Yv^{-2n}\left(\sum_{k=0}^{n-6}v^{-2k}\right)+(n-5)Zv^{-4n}+v^{-4(n-6)}F_{K_5}.
\end{align*}Using sage we compute $F_{K_5}$ to be
{\scriptsize
\begin{align*}
	F_{K_5}=&\,1+8v^{-4}-36v^{-6}+59v^{-8}-41v^{-10}+24v^{-12}-7v^{-14}-242v^{-16}+755v^{-18}\\
	&-958v^{-20}+569v^{-22}-131v^{-24}
\end{align*}}and by plugging in the values of $F_{K_5}$, $Y$, and $Z$ we see that
\begin{align*}
	F_{K_n}=1+(2n-2)v^{6-2n}+\text{ lower order terms}
\end{align*}
for $n\geq6$. For the other positive values of $n$ (i.e.\ for $n=1,2,3,4,5$) we directly compute $F$ using sage~\cite{data}. 

For the knots $K'_n$, we perform the same steps. Here we get
for all $n\geq6$ the values of $F_{J'_i}$ shown in Table~\ref{tab:tab_homfly_values}
from which we compute
{\scriptsize
	\begin{align*}
		X'=&\,1-v^{-4}\\
		Y'=&\,-(1+n)v^{-6}+(4+4n)v^{-4}+(-5n-8)v^{-2}+12+(5n-13)v^{2}+(8-4n)v^{4}+(n-2)v^{6}\\
		Z'=&\,-(2+7n)v^{-4}+32nv^{-2}+(17-55n)+(40n-28)v^{2}+(12-5n)v^{4}+(4-8n)v^{6}+(3n-3)v^{8}.
\end{align*}}

Using sage we compute $F_{K'_5}$ to be
{\scriptsize
	\begin{align*}
    F_{K'_5}=&\,1+3v^{-4}-10v^{-6}+5v^{-8}+15v^{-10}+10v^{-12}-91v^{-14}-18v^{-16}+459v^{-18}\\
		&-739v^{-20}+483v^{-22}-117v^{-24}
\end{align*}}and thus we can deduce from Lemma~\ref{lem:ana1} and the computations of $F_{K'_n}$ for small $n$, the claimed statements about $F_{K'_n}$.
\end{proof}

\subsection{Inductive formulas for 2, 3, and 4 twists}
For the proof of Proposition \ref{prop:Jipolys} we will need the following inductive formulas, which are all similar to Lemma~\ref{lem:5twist}, but for $2$, $3$, and $4$ twisting strands. 

\begin{lem}\label{lem:2twist_one_up_one_down}
	For $n\in\N_0$, let $K_n$ be the infinite family of knots that are outside of a $3$-ball $B$ identical and inside $B$ have $2$ strands performing $n$-full twists, oriented in opposite directions as shown Figure~\ref{fig:2strands_up_down}. Let $Z$ be the $2$-component link that is outside of $B$ the same as $K_n$ but inside $B$ as shown in Figure~\ref{fig:2strands_up_down}.
	For $n\geq1$, their $F$ polynomials are related by
	\begin{align*}
		F_{K_n}=\,& -F_Z \left(\sum_{k=0}^{n-1}v^{-2k}\right)+ v^{-2n} F_{K_{0}}
		= -F_Z \frac{v^{-2n}-1}{v^{-2}-1}+ v^{-2n} F_{K_{0}}.
	\end{align*}	
\end{lem}

\begin{figure}[t] 
	\centering
	\def\svgwidth{0,4\columnwidth}
\begingroup%
  \makeatletter%
  \providecommand\color[2][]{%
    \errmessage{(Inkscape) Color is used for the text in Inkscape, but the package 'color.sty' is not loaded}%
    \renewcommand\color[2][]{}%
  }%
  \providecommand\transparent[1]{%
    \errmessage{(Inkscape) Transparency is used (non-zero) for the text in Inkscape, but the package 'transparent.sty' is not loaded}%
    \renewcommand\transparent[1]{}%
  }%
  \providecommand\rotatebox[2]{#2}%
  \newcommand*\fsize{\dimexpr\f@size pt\relax}%
  \newcommand*\lineheight[1]{\fontsize{\fsize}{#1\fsize}\selectfont}%
  \ifx\svgwidth\undefined%
    \setlength{\unitlength}{200.35511389bp}%
    \ifx\svgscale\undefined%
      \relax%
    \else%
      \setlength{\unitlength}{\unitlength * \real{\svgscale}}%
    \fi%
  \else%
    \setlength{\unitlength}{\svgwidth}%
  \fi%
  \global\let\svgwidth\undefined%
  \global\let\svgscale\undefined%
  \makeatother%
  \begin{picture}(1,0.33498648)%
    \lineheight{1}%
    \setlength\tabcolsep{0pt}%
    \put(0,0){\includegraphics[width=\unitlength,page=1]{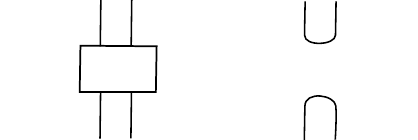}}%
    \put(0.26092541,0.15309196){\color[rgb]{0,0,0}\makebox(0,0)[lt]{\lineheight{0.1}\smash{\begin{tabular}[t]{l}$n$\end{tabular}}}}%
    \put(-0.00371341,0.1450478){\color[rgb]{0,0,0}\makebox(0,0)[lt]{\lineheight{0.1}\smash{\begin{tabular}[t]{l}$K_{n}$\end{tabular}}}}%
    \put(0.90057647,0.14082033){\color[rgb]{0,0,0}\makebox(0,0)[lt]{\lineheight{0.1}\smash{\begin{tabular}[t]{l}$Z$\end{tabular}}}}%
    \put(0,0){\includegraphics[width=\unitlength,page=2]{2strands_up_down.pdf}}%
  \end{picture}%
\endgroup%

	\caption{The knot $K_n$ and the $2$-component link $Z$.}
	\label{fig:2strands_up_down}
\end{figure}

\begin{proof}
	Any crossing $c$ in the twist region of $K_n$ is negative. By changing $c$ we get $K_{n-1}$ and its resolution is isotopic to $Z$. Thus the skein relation of $F$ yields
	\begin{align*}
		F_{K_n}=-F_Z+v^{-2}F_{K_{n-1}}
	\end{align*}
	and then induction implies the claimed formula.
\end{proof}

\begin{lem}\label{lem:2twist_all_up}
		For $n\in\N_0$, let $K_n$ be the infinite family of knots that are outside of a $3$-ball $B$ identical and inside $B$ have $2$ strands performing $n$-full twists, oriented in the same direction as shown in Figure~\ref{fig:2strands_all_up}. Let $X\cup Y$ be the $2$-component link with components $X$ and $Y$ that is outside of $B$ the same as $K_n$ but inside $B$ as shown in Figure~\ref{fig:2strands_all_up}.
	For $n\geq1$, their $F$ polynomials are related by
	\begin{align*}
		F_{K_n}=\,& (v^{-2}-1) n v^{2n} v^{2\lk(X,Y)} F_X F_Y+ v^{2n} F_{K_{0}}.
	\end{align*}	
\end{lem}

\begin{figure}[t] 
	\centering
	\def\svgwidth{0,4\columnwidth}
\begingroup%
  \makeatletter%
  \providecommand\color[2][]{%
    \errmessage{(Inkscape) Color is used for the text in Inkscape, but the package 'color.sty' is not loaded}%
    \renewcommand\color[2][]{}%
  }%
  \providecommand\transparent[1]{%
    \errmessage{(Inkscape) Transparency is used (non-zero) for the text in Inkscape, but the package 'transparent.sty' is not loaded}%
    \renewcommand\transparent[1]{}%
  }%
  \providecommand\rotatebox[2]{#2}%
  \newcommand*\fsize{\dimexpr\f@size pt\relax}%
  \newcommand*\lineheight[1]{\fontsize{\fsize}{#1\fsize}\selectfont}%
  \ifx\svgwidth\undefined%
    \setlength{\unitlength}{185.95892026bp}%
    \ifx\svgscale\undefined%
      \relax%
    \else%
      \setlength{\unitlength}{\unitlength * \real{\svgscale}}%
    \fi%
  \else%
    \setlength{\unitlength}{\svgwidth}%
  \fi%
  \global\let\svgwidth\undefined%
  \global\let\svgscale\undefined%
  \makeatother%
  \begin{picture}(1,0.36046362)%
    \lineheight{1}%
    \setlength\tabcolsep{0pt}%
    \put(0,0){\includegraphics[width=\unitlength,page=1]{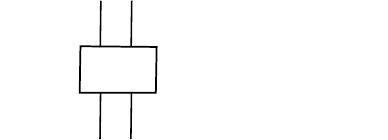}}%
    \put(0.2811252,0.16289007){\color[rgb]{0,0,0}\makebox(0,0)[lt]{\lineheight{0.1}\smash{\begin{tabular}[t]{l}$n$\end{tabular}}}}%
    \put(-0.00400088,0.15422316){\color[rgb]{0,0,0}\makebox(0,0)[lt]{\lineheight{0.1}\smash{\begin{tabular}[t]{l}$K_{n}$\end{tabular}}}}%
    \put(0.89326672,0.04763){\color[rgb]{0,0.58431373,0}\makebox(0,0)[lt]{\lineheight{0.1}\smash{\begin{tabular}[t]{l}$Y$\end{tabular}}}}%
    \put(0.88264593,0.28524368){\color[rgb]{1,0,0}\makebox(0,0)[lt]{\lineheight{0.1}\smash{\begin{tabular}[t]{l}$X$\end{tabular}}}}%
    \put(0,0){\includegraphics[width=\unitlength,page=2]{2strands_all_up.pdf}}%
  \end{picture}%
\endgroup%

	\caption{The knot $K_n$ and the $2$-component link $X\cup Y$.}
	\label{fig:2strands_all_up}
\end{figure}

\begin{proof}
	Any crossing $c$ in the twist region of $K_n$ is positive. By changing $c$ we get $K_{n-1}$ and its resolution is isotopic to the $2$-component link $X\cup Y$ with an additional $(n-1)$-twists between the components $X$ and $Y$. Thus the skein relation of $F$ and the linking number formula yields
	\begin{align*}
		F_{K_n}=(v^{-2}-1)v^{2n}v^{2\lk(X,Y)}F_XF_Y+v^{2}F_{K_{n-1}}
	\end{align*}
	and then induction implies the claimed formula.
\end{proof}

\begin{lem}\label{lem:3twist}
	For $n\in\N_0$, let $K_n$ be the infinite family of knots that are outside of a $3$-ball $B$ identical and inside $B$ have $3$ strands performing $n$-full twists, oriented as shown in Figure~\ref{fig:3strands}. Let $U$, $V$ and $L_{n-1}$ be the links that are outside of $B$ the same as $K_n$ but are inside $B$ as shown in Figure~\ref{fig:3strands}.
	For $n\geq1$, their $F$ polynomials are related by
	\begin{align*}
		F_{K_n}=\,&v^2 F_{L_{n-1}} -v^2 F_U -F_V + v^{-2} F_{K_{n-1}}.
	\end{align*}	
\end{lem}

\begin{figure}[t] 
	\centering
	\def\svgwidth{0,55\columnwidth}
	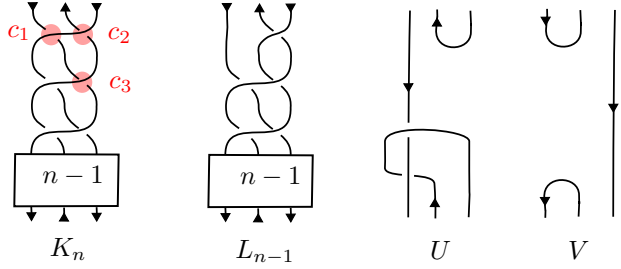
	\caption{The knot $K_n$ with three marked crossings $c_1$, $c_2$, and $c_3$ and the $2$-component links $L_{n-1}$, $U$, and $V$.}
	\label{fig:3strands}
\end{figure}

\begin{rem}
	In general, the link $L_{n-1}$ depends again on $n$. But since $L_{n-1}$ is a $2$-component link, we can use the linking number formula to reduce the computation of $F_{L_{n-1}}$ to the computation of $F$ of the components of $L_{n-1}$, which can be done with Lemmas~\ref{lem:2twist_all_up} and~\ref{lem:2twist_one_up_one_down}. We refer to Section~\ref{sec:J1} for examples.
\end{rem}

\begin{proof}[Proof of Lemma~\ref{lem:3twist}]
	For the crossings $c_1$, $c_2$, and $c_{3}$ in $K_n$ as marked in Figure~\ref{fig:3strands} we define for $i=1,2,3$ the knot $K_n^i$ to be the knot obtained from $K_n$ by changing all crossings $c_j$ for $j\leq i$. We observe that $K_n^{3}=K_{n-1}$ and we set $K_n^0=K_n$.
	Moreover, we check similarly as in the proof of Lemma~\ref{lem:5twist} that the resolutions of these crossings yield the links $L_{n-1}$, $U$, and $V$. Then we get by successively applying the skein relation of $F$ to the crossing $c_{i+1}$ in $K_n^i$, the claimed formula.
\end{proof}

\begin{lem}\label{lem:4twist}
	For $n\in\N_0$, let $K_n$ be the infinite family of knots that are outside of a $3$-ball $B$ identical and inside $B$ have $4$ strands performing $n$-full twists, oriented as shown in Figure~\ref{fig:4strands}. For $i=1,\ldots, 6$, let $I_i$ be the links that are outside of $B$ the same as $K_n$ but are inside $B$ as shown in Figure~\ref{fig:4strands}.
	For $n\geq1$, their $F$ polynomials are related by
	\begin{align*}
		F_{K_n}=\,&-F_{I_1} + F_{I_2} -F_{I_3} +F_{I_4}  -F_{I_5}  - v^{-2} F_{I_{6}} + v^{-4} F_{K_{n-1}}.
	\end{align*}	
\end{lem}

\begin{figure}[t] 
	\centering
	\def\svgwidth{0,9\columnwidth}
	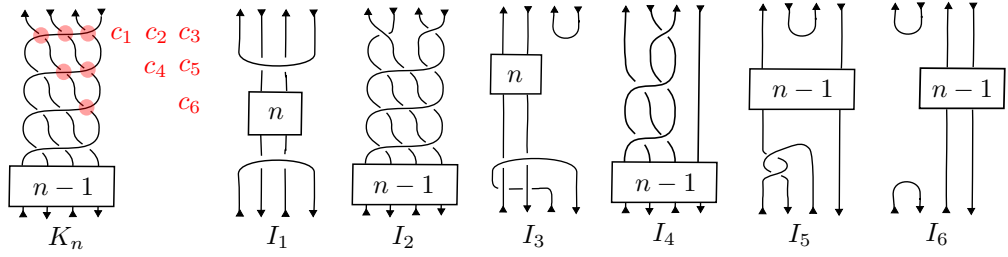
	\caption{The knot $K_n$ with six marked crossings and the $2$-component links $I_1,\ldots ,I_6$.}
	\label{fig:4strands}
\end{figure}

\begin{proof}
	As in the proof of Lemma~\ref{lem:5twist} and~\ref{lem:3twist}, we consider the crossings $c_1,\ldots, c_{6}$ in $K_n$ as marked in Figure~\ref{fig:4strands}. Then we check that with the same notation as there, the links $I_i$ are given by the resolutions and thus the skein relation implies the formula for $F$.
\end{proof}

\subsection{Computation of \texorpdfstring{$J_1$}{J1}}\label{sec:J1}

Using the formulas from Lemmas~\ref{lem:2twist_one_up_one_down}-\ref{lem:4twist} we can compute the $F$ polynomials of the links $J_i$ and $J'_i$. We have suppressed the computations for $i>1$; 
they work analogously. In the rest of this section, $n$ is always a positive number.

\begin{lem}
	The family of links $J_1=J_1^n$ shown in Figure~\ref{fig:J1A1} have 
	{\scriptsize
	\begin{align*}
		F_{J_1}=&\,(v^{-2}-1)\big( 1 + v^{-2n+2} -v^{-2n} -v^{-2n-2} + v^{-2n-4} -(n-2) v^{-4n+6} +(n-3)v^{-4n+4} \\&+(6n-2) v^{-4n+2} -(14n-4) v^{-4n} +11nv^{-4n-2} -(3n+1)v^{-4n-4}  \big).
	\end{align*}}
\end{lem}

\begin{figure}[t] 
	\centering
	\def\svgwidth{0,65\columnwidth}
\begingroup%
  \makeatletter%
  \providecommand\color[2][]{%
    \errmessage{(Inkscape) Color is used for the text in Inkscape, but the package 'color.sty' is not loaded}%
    \renewcommand\color[2][]{}%
  }%
  \providecommand\transparent[1]{%
    \errmessage{(Inkscape) Transparency is used (non-zero) for the text in Inkscape, but the package 'transparent.sty' is not loaded}%
    \renewcommand\transparent[1]{}%
  }%
  \providecommand\rotatebox[2]{#2}%
  \newcommand*\fsize{\dimexpr\f@size pt\relax}%
  \newcommand*\lineheight[1]{\fontsize{\fsize}{#1\fsize}\selectfont}%
  \ifx\svgwidth\undefined%
    \setlength{\unitlength}{484.03889336bp}%
    \ifx\svgscale\undefined%
      \relax%
    \else%
      \setlength{\unitlength}{\unitlength * \real{\svgscale}}%
    \fi%
  \else%
    \setlength{\unitlength}{\svgwidth}%
  \fi%
  \global\let\svgwidth\undefined%
  \global\let\svgscale\undefined%
  \makeatother%
  \begin{picture}(1,0.55164843)%
    \lineheight{1}%
    \setlength\tabcolsep{0pt}%
    \put(0,0){\includegraphics[width=\unitlength,page=1]{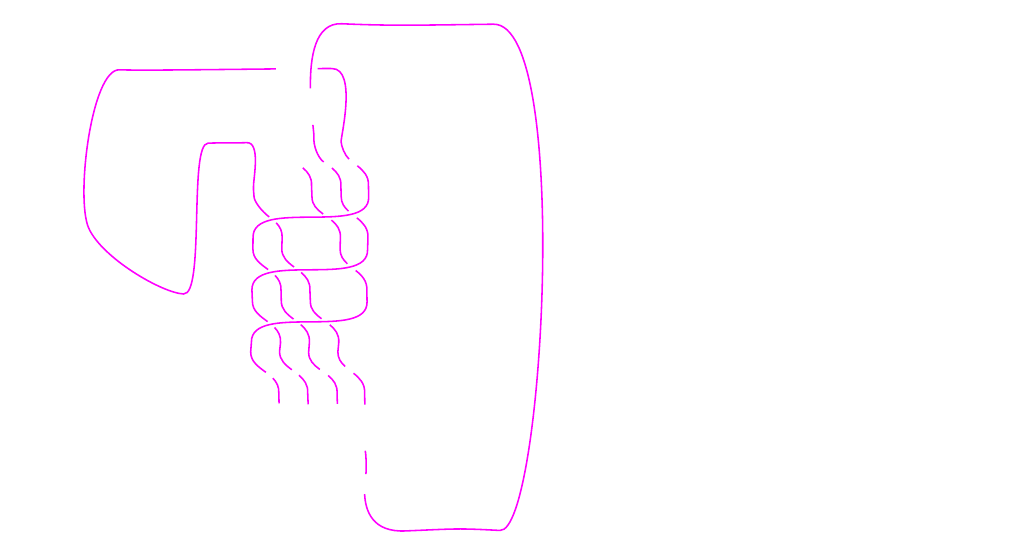}}%
    \put(0.27369468,0.11931621){\color[rgb]{0,0,0}\makebox(0,0)[lt]{\lineheight{0.1}\smash{\begin{tabular}[t]{l}$n-1$\end{tabular}}}}%
    \put(-0.00076234,0.26031712){\color[rgb]{0,0,0}\makebox(0,0)[lt]{\lineheight{0.1}\smash{\begin{tabular}[t]{l}$J_{1}^n$\end{tabular}}}}%
    \put(0,0){\includegraphics[width=\unitlength,page=2]{J1A1.pdf}}%
    \put(0.77593807,0.24099033){\color[rgb]{1,0,1}\makebox(0,0)[lt]{\lineheight{0.1}\smash{\begin{tabular}[t]{l}$n$\end{tabular}}}}%
    \put(0.78876571,0.07749528){\color[rgb]{1,0,1}\makebox(0,0)[lt]{\lineheight{0.1}\smash{\begin{tabular}[t]{l}$A^{1}_n$\end{tabular}}}}%
    \put(0.40172446,0.48559141){\color[rgb]{1,0,1}\makebox(0,0)[lt]{\lineheight{0.1}\smash{\begin{tabular}[t]{l}$A^{1}_n$\end{tabular}}}}%
    \put(0.13623708,0.52298673){\color[rgb]{1,0.58039216,0}\makebox(0,0)[lt]{\lineheight{0.1}\smash{\begin{tabular}[t]{l}$B^{1}$\end{tabular}}}}%
    \put(0,0){\includegraphics[width=\unitlength,page=3]{J1A1.pdf}}%
  \end{picture}%
\endgroup%

	\caption{Left: the $2$-component link $J_1^n$, consisting of the two components $B^1$ and $A^1_n$. Right: A diagram of $A^1_n$.}
	\label{fig:J1A1}
\end{figure}

\begin{proof}
	In Figure~\ref{fig:J1A1} we see that $J_1^n$ consists of two components: the component $B^1$ which is an unknot and the component $A^1_n$ shown in Figure~\ref{fig:J1A1}. Their linking number is $\lk(A^1_n,B^1)=0$, and thus the linking number formula implies that $F_{J_1^n}=(v^{-2}-1)F_{A^1_n}$. Therefore, the claim follows from the following lemma. 
\end{proof}

\begin{lem}
	The family of knots $A_1^n$ shown in Figure~\ref{fig:J1A1} have 
	{\scriptsize
		\begin{align*}
			F_{A_1^n}=&\,\ 1 + v^{-2n+2} -v^{-2n} -v^{-2n-2} + v^{-2n-4} -(n-2) v^{-4n+6} +(n-3)v^{-4n+4} \\&+(6n-2) v^{-4n+2} -(14n-4) v^{-4n} +11nv^{-4n-2} -(3n+1)v^{-4n-4}  .
	\end{align*}}
\end{lem}

\begin{proof}
	First, we use sage to compute 
	\begin{align*}
		F_{A^1_0}=\,&2v^6-3v^4-v^2+4-v^{-2}.
	\end{align*}	
	Then we can apply Lemma~\ref{lem:4twist} to the family of knots $A_1^n$ to see that
		\begin{align*}
		F_{A^1_n}=\,&-F_{I_1} + F_{I_2} -F_{I_3} +F_{I_4}  -F_{I_5}  - v^{-2} F_{I_{6}} + v^{-4} F_{A^1_{n-1}}.
	\end{align*}	
In Lemmas~\ref{lem:I1}--\ref{lem:I6} we compute $F_{I_i}$. If we plug in the values for $F_{I_i}$ we get 
\begin{align*}
	F_{A^1_n}=\,&X+v^{-2n}Y+v^{-4n}Z + v^{-4} F_{A^1_{n-1}},
\end{align*}	
with polynomials	
	\begin{align*}
		X&=1-v^{-4}\\
		Y&=v^2-2+2v^{-4}-v^{-6}\\
		Z&=-v^6+v^4+6v^2-14+11v^{-2}-3v^{-4}.
\end{align*}
Thus Lemma~\ref{lem:ana1} implies that 
	\begin{align*}
		F_{A^1_n}=1-v^{-4n}+Yv^{-2n}\frac{1-v^{-2n}}{1-v^{-2}}+nZv^{-4n}+v^{-4n}F_{A^1_0}.
\end{align*}
Plugging in the above values yields the claimed formula for $F_{A^1_n}$.	
\end{proof}

\begin{lem} \label{lem:I1}
	Let $I^n_1$ be the family of knots shown in Figure~\ref{fig:Ii} then we have 
		\begin{align*}
			F_{I^n_1}=(v^{-2}-1)v^{-2n-2}\,\big( v^2-1+v^{-2}  \big).
	\end{align*}
\end{lem}

\begin{proof}
	In Figure~\ref{fig:Ii} we see that $I_1$ consists of two knots $X^1$ and $Y^1$, that have linking number $\lk(X^1,Y^1)=-n-1$. It is also straightforward to check that $X^1$ is an unknot and $Y^1$ is the figure eight knot with $F_{Y^1}=v^2-1-v^{-2}$. Thus the linking number formula for $F$ implies the result.
\end{proof}

\begin{lem}
	Let $I^n_2$ be the family of knots shown in Figure~\ref{fig:Ii} then we have 
		\begin{align*}
			F_{I^n_2}=(v^{-2}-1)v^{2-2n}\,\big( 1+v^{-2n}(2v^2-6+6v^{-2}-2v^{-4})\big).
	\end{align*}
\end{lem}

\begin{proof}
	From Figure~\ref{fig:Ii} we deduce that $I^n_2$ consist of an unknot $Y^2$ and a knot $X^2_n$ that have linking number $\lk(X^2_n,Y^2)=1-n$. Thus $F_{I^n_2}=(v^{-2}-1)v^{2-2n}F_{X^2_n}$. The $F$ polynomial of $X_n^2$ can be computed using Lemma~\ref{lem:3twist}. In the notation of that lemma, we see that $L_{n-1}$ is a $2$-component link consisting of two unknots that have vanishing linking numbers and thus, independent of $n$, we have $F_{L_{n-1}}=v^{-2}-1$. Moreover, we compute 
	\begin{align*}
		F_U&=(v^{-2}-1)v^2(v^{-2}+v^{-4}-v^{-6})\\
		F_V&=(v^{-2}-1)v^{-2}.
	\end{align*}
Thus Lemma~\ref{lem:3twist} implies that 
\begin{align*}
	F_{X^2_n}=-(v^{-2}-1)+v^{-2}F_{X_{n-1}^2}
\end{align*}
and induction yields
\begin{align*}
	F_{X^2_n}=1-v^{-2n}+v^{-2n}F_{X_0^2}.
\end{align*}
By computing $F_{X^2_0}=2v^2-5+6v^{-2}-2v^{-4}$ and plugging everything into the formula for $F_{I_2}$ we get the claimed result.
\end{proof}

\begin{lem}
	Let $I^n_3$ be the family of knots shown in Figure~\ref{fig:Ii} then we have 
		\begin{align*}
			F_{I^n_3}=(v^{-2}-1)\,\big( 1+v^{-2n}(v^2-1-v^{-2}+v^{-4})  \big).
	\end{align*}
\end{lem}

\begin{proof}
	In Figure~\ref{fig:Ii} we observe that $I_3$ consists of an unknot $Y^3$ that has vanishing linking number with the other component $X^3_n$ and thus $F_{I_3^n}=(v^{-2}-1)F_{X_n^3}$. We compute $F_{X^3_n}$ with the help of Lemma~\ref{lem:2twist_one_up_one_down}. In the notation of that lemma, we observe that $Z$ is a $2$-component unlink with $F_Z=v^{-2}-1$ and we compute $F_{X_0^3}=v^2-v^{-2}+v^{-4}$. Putting all this together yields the claimed result.
\end{proof}

\begin{lem}
	Let $I^n_4$ be the family of knots shown in Figure~\ref{fig:Ii} then we have 
		\begin{align*}
			F_{I^n_4}=(v^{-2}-1)v^{2-2n}\,\big( (v^{2-2n}-1)(-v^2+v^{-4}-v^{-6})+v^{2-2n}(2v^2-2+v^{-4})  \big).
	\end{align*}
\end{lem}

\begin{proof}
	In Figure~\ref{fig:Ii} we observe that $I^n_4$ consist of an unknot $Y^4$ and a knot $X^4_n$ that have linking number $\lk(X^4_n,Y^4)=1-n$. Thus $F_{I^n_4}=(v^{-2}-1)v^{2-2n}F_{X^4_n}$. The $F$ polynomial of $X_n^4$ can be computed using Lemma~\ref{lem:3twist}. In the notation of that lemma, we see that $L_{n-1}$ is a $2$-component link consisting of two unknots that have linking number $\lk=-1$ and thus independent of $n$, we have $F_{L_{n-1}}=(v^{-2}-1)v^{-2}$. Moreover, we compute 
	\begin{align*}
		F_U&=(v^{-2}-1)v^{-4}(v^{2}-v^{-2}+v^{-4})\\
		F_V&=(v^{-2}-1)v^{2}.
	\end{align*}
	Thus Lemma~\ref{lem:3twist} implies that 
	\begin{align*}
		F_{X^4_n}=(v^{-2}-1)(-v^2+v^{-4}-v^{-6})+v^{-2}F_{X_{n-1}^4}
	\end{align*}
	and induction yields
	\begin{align*}
		F_{X^4_n}=(v^{-2n}-1)(-v^2+v^{-4}-v^{-6})+v^{-2n}F_{X_0^4}.
	\end{align*}
	By computing $F_{X^4_0}=2v^2-2+v^{-4}$ and plugging everything into the formula for $F_{I_4}$ we get the claimed result.
\end{proof}

\begin{lem}
	Let $I^n_5$ be the family of knots shown in Figure~\ref{fig:Ii} then we have 
		\begin{align*}
			F_{I^n_5}=(v^{-2}-1)v^{6-2n}\,\big( v^{-2}+v^{-4}-v^{-6}  \big).
	\end{align*}
\end{lem}

\begin{proof}
	The components $X^5$ and $Y^5$ of the link $I_5^n$ have linking number $\lk(X^5,Y^5)=3-n$. $Y^5$ is an unknot, while $F_{X^5}=v^{-2}+v^{-4}-v^{-6}$ and thus the linking number formula for $F$ implies the lemma.
\end{proof}

\begin{lem} \label{lem:I6}
	Let $I^n_6$ be the family of knots shown in Figure~\ref{fig:Ii} then we have 
		\begin{align*}
			F_{I^n_6}=v^{-2}-1.
	\end{align*}
\end{lem}

\begin{proof}
	$I^n_6$ is a $2$-component unlink and thus $F_{I^n_6}=v^{-2}-1$.
\end{proof}

\begin{figure}[ht] 
	\centering
	\def\svgwidth{0,75\columnwidth}
	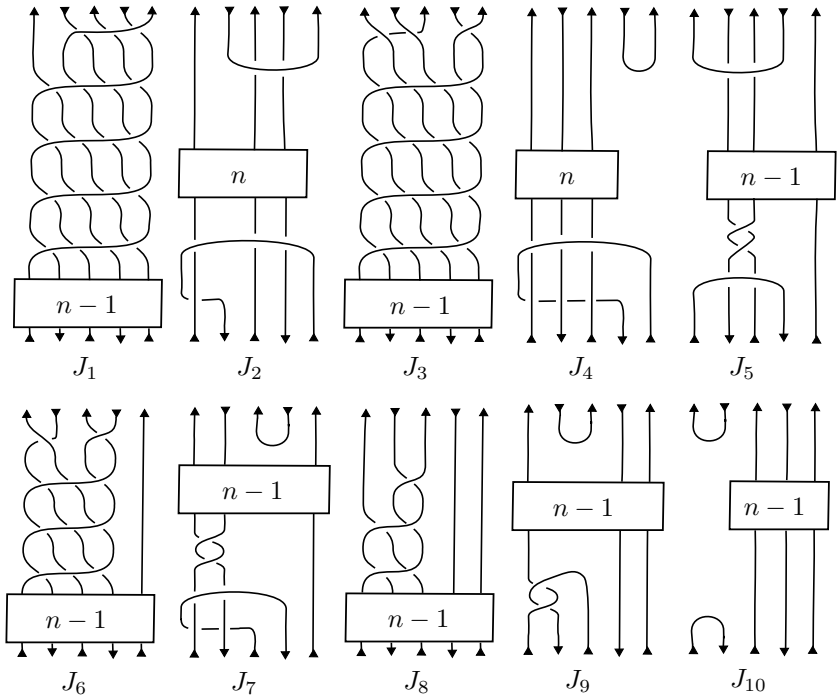
	\caption{The $2$-component links $J_i$.}
	\label{fig:Ji}
\end{figure}

\begin{figure}[ht] 
	\centering
	\def\svgwidth{0,8\columnwidth}
	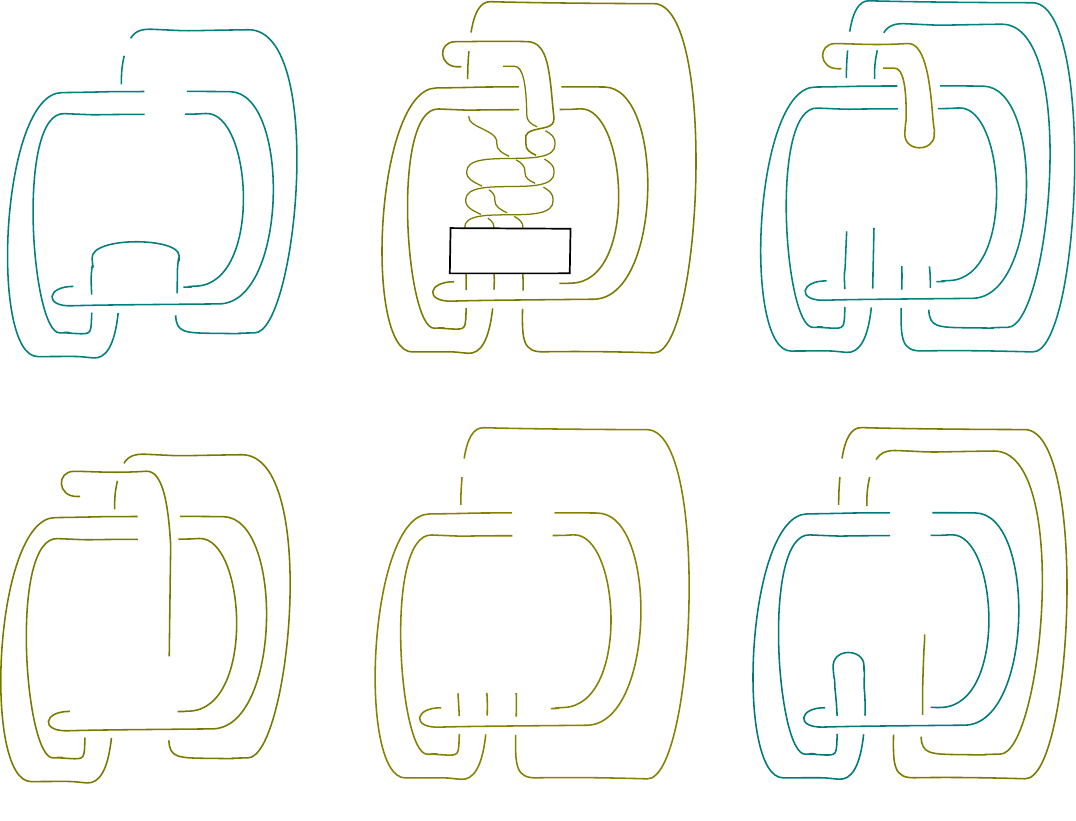
	\caption{The $2$-component links $I_i$.}
	\label{fig:Ii}
\end{figure}

{\scriptsize
\begin{table}[ht]
\begin{tabular}{ rl} 
 \midrule
  \midrule
	$F_{K_{-2}}$& 
	$=8v^{16} - 8v^{14} - 46v^{12} + 105v^{10} - 77v^8 + 12v^6 + 10v^4 - 5v^2 + 2$\\
	$F_{K'_{-2}}$&$=
	9v^{16} - 14v^{14} - 32v^{12} + 91v^{10} - 77v^8 + 26v^6 - 4v^4 + v^2 + 1$\\
    \midrule
	$F_{K_{-1}}=F_{K'_-1}$&$=
	3v^{12} - 2v^{10} - 18v^8 + 35v^6 - 23v^4 + 6v^2  - 1+ v^{-2}$\\
    \midrule
	$F_{K_0}=F_{K'_0}$&$=
	2v^4 - 6v^2  + 7- 2v^{-2}$\\
    \midrule
	$F_{K_1}=F_{K'_1}$&$=
	-v^4 + 2+ 18v^{-2} - 45v^{-4} + 38v^{-6} - 11v^{-8} $\\
    \midrule
	$F_{K_2}$&$=
	2v^2 - 11+ 27v^{-2} - 55v^{-4} + 122v^{-6} - 168v^{-8} + 113v^{-10} - 29v^{-12}$ \\
	$F_{K'_2}$&$=
	 3-15v^{-2} + 15v^{-4} + 52v^{-6} - 126v^{-8} + 99v^{-10} - 27v^{-12} $\\
     \midrule
	$F_{K_3}$&$=
	 5-20v^{-2} + 38v^{-4} - 23v^{-6} - 81v^{-8} + 272v^{-10} - 362v^{-12} + 227v^{-14} - 55v^{-16} $\\
	$F_{K'_3}$&$=
	 2-4v^{-2} + 10v^{-4} - 23v^{-6} - 11v^{-8} + 160v^{-10} - 278v^{-12} + 195v^{-14} - 50v^{-16} $\\
        \midrule
           \midrule
	$F_{J_1}$&$=\,(v^{-2}-1)\big( 1 + v^{-2n+2} -v^{-2n} -v^{-2n-2} + v^{-2n-4} -(n-2) v^{-4n+6} +(n-3)v^{-4n+4} $\\&$\,\,\,\,\,\,\,+(6n-2) v^{-4n+2} -(14n-4) v^{-4n} +11nv^{-4n-2} -(3n+1)v^{-4n-4}  \big)$\\
    \midrule
	$F_{J_2}$&$=\,(v^{-2}-1)\big( v^{2} -1 +v^{-2}  \big)$\\
    \midrule
	$F_{J_3}$&$=\,(v^{-2}-1)v^2\big( 1 -2v^{-2n+2} +6v^{-2n} -6v^{-2n-2} + 2v^{-2n-4} -(n-2) v^{-4n+2} +(8n-5)v^{-4n}$ \\&$\,\,\,\,\,\,\,-(12n-3) v^{-4n-2} +(8n+1) v^{-4n-4} -(2n+1)v^{-4n-6}   \big)$\\
    \midrule
	$F_{J_4}$&$=\,(v^{-2}-1)\big( 1 +v^{-2n+2} -v^{-2n} -v^{-2n-2} + v^{-2n-4}   \big)$\\
    \midrule
	$F_{J_5}$&$=\,(v^{-2}-1)\big( v^2 -1 +v^{-2} +(2n-4)v^{-2n+6} -(6n-9)v^{-2n+4} +(6n-6)v^{-2n+2} -(2n-1) v^{-2n}    \big)$\\
    \midrule
	$F_{J_6}$&$=\,(v^{-2}-1)v^{2-4n}\big( -(n+1)v^{2n+4} + (2n-2) v^{2n+2} -(2n-3) v^{2n}+(n-2) v^{2n-2} + v^{2n-4}$\\&$\,\,\,\,\,\,\,-(n-1)v^6+(3n-2)v^4-2nv^2-(2n-2)+(3n-1)v^{-2}-nv^{-4}\big)$\\
    \midrule
	$F_{J_7}$&$=\,(v^{-2}-1)v^{8-4n}\big( -(n-1)v^{2n} + n v^{2n-2}\big)\big( v^{-2}+v^{-4}-v^{-6}\big)$\\
    \midrule
	$F_{J_8}$&$=\,(v^{-2}-1)v^{4-4n}\big( -(n-1)v^{2n} + n v^{2n-2}\big)\big( 2v^{-2}-v^{-4}+2v^{-2n+2}-6v^{-2n}+6v^{-2n-2}-2v^{-2n-4}\big)$\\
    \midrule
	$F_{J_9}$&$=\,(v^{-2}-1)$\\
    \midrule
	$F_{J_{10}}$&$=\,(v^{-2}-1)v^{-4n}\big( -v^{2n+4} - (n-3) v^{2n+2} +(n-2)v^{2n}+v^{2n-2}$\big)\\
        \midrule
            \midrule
    	$F_{J'_1}$&$=\,(v^{-2}-1)\big( 1 -(n-1) v^{-4n+6} +(n-1)v^{-4n+4} +6n v^{-4n+2} -(14n+4) v^{-4n}$ \\&$\,\,\,\,\,\,\,+(11n+7)v^{-4n-2} -(3n+3)v^{-4n-4}  \big)$\\
    \midrule
	$F_{J'_2}$&$=\,(v^{-2}-1)\big( v^{-2} + v^{-2n}-v^{-2n-2}  \big)$\\
    \midrule
	$F_{J'_3}$&$=\,(v^{-2}-1)v^2\big( v^{-2} +v^{-2n} -2v^{-2n-4} +v^{-2n-6} -n v^{-4n+4} +(3n+1) v^{-4n+2} -(2n+1)v^{-4n}$ \\&$\,\,\,\,\,\,\,-(2n+3) v^{-4n-2} +(3n+5) v^{-4n-4} -(n+2)v^{-4n-6}   \big)$\\
    \midrule
	$F_{J_4'}$&$=\,(v^{-2}-1)\big( 1 +v^{-2n+4} -v^{-2n+2} -v^{-2n} + v^{-2n-2}   \big)$\\
    \midrule
	$F_{J_5'}$&$=\,(v^{-2}-1)\big( v^2 -1 +v^{-2} +(n-3)v^{-2n+6} -(3n-8)v^{-2n+4} +(2n-7)v^{-2n+2} +(2n+3) v^{-2n}$\\& $\,\,\,\,\,\,\,-(3n+2)v^{-2n-2}+(n+1)v^{-2n-4}    \big)$\\
    \midrule
	$F_{J'_6}$&$=\,(v^{-2}-1)v^{2-4n}\big( -v^{2n+4} + 3 v^{2n+2} -(n+2) v^{2n}+(n+1) v^{2n-2} - (2n-4)v^{4}$\\&$\,\,\,\,\,\,\,+(8n-14)v^2-(12n-18)+(8n-10)v^{-2}-(2n-2)v^{-4}\big)$\\
    \midrule
	$F_{J'_7}$&$=\,(v^{-2}-1)v^{8-4n}\big( -(n-2)v^{2n-2} + (n-1) v^{2n-4}\big)\big( 2v^{-2}-v^{-4}\big)$\\
    \midrule
	$F_{J'_8}$&$=\,(v^{-2}-1)v^{8-4n}\big( -(n-2)v^{2n-2} + (n-1) v^{2n-4}\big)\big( 2v^{-2}-v^{-4}+v^{-2n+2}-2v^{-2n}$\\&$\,\,\,\,\,\,\,+2v^{-2n-4}-v^{-2n-6}\big)$\\
    \midrule
	$F_{J'_9}$&$=\,(v^{-2}-1)$\\
    \midrule
	$F_{J'_{10}}$&$=\,(v^{-2}-1)v^{4-4n}\big( -v^{2n+4}  +3 v^{2n+2} -(n+2)v^{2n}+(n+1)v^{2n-2}\big)$\\
 \midrule \midrule
\end{tabular}
    \caption{The $F$ polynomials of the knots $K_n$ and $K_n'$ for small values of $n$, and for the links $J_i$ and $J_i'$.}
    \label{tab:tab_homfly_values}
\end{table} }

\let\MRhref\undefined
\bibliographystyle{hamsalpha}
\bibliography{friends.bib}

\end{document}